\documentclass[11pt,reqno,a4paper]{amsart}
\usepackage[T1]{fontenc}
\usepackage[utf8]{inputenc}
\usepackage{lmodern}
\usepackage{amsmath,amssymb,amsthm,mathtools,mathrsfs}
\usepackage{geometry}
\usepackage{microtype,enumitem,booktabs,longtable,needspace}
\usepackage[dvipsnames]{xcolor}
\usepackage[colorlinks=true,linkcolor=MidnightBlue,citecolor=MidnightBlue,urlcolor=MidnightBlue]{hyperref}
\setlist[enumerate]{leftmargin=*,topsep=3pt,itemsep=3pt}
\setlist[itemize]{leftmargin=*,topsep=3pt,itemsep=3pt}
\numberwithin{equation}{section}
\newtheorem{theorem}{Theorem}[section]
\newtheorem{mainresult}{Theorem}

\newtheorem{maincorollary}[mainresult]{Corollary}
\newtheorem{proposition}[theorem]{Proposition}
\newtheorem{corollary}[theorem]{Corollary}
\newtheorem{lemma}[theorem]{Lemma}
\theoremstyle{definition}
\newtheorem{definition}[theorem]{Definition}

\newtheorem{example}[theorem]{Example}
\theoremstyle{remark}
\newtheorem{remark}[theorem]{Remark}
\newcommand{\C}{\mathbb C}
\newcommand{\R}{\mathbb R}
\newcommand{\Z}{\mathbb Z}
\newcommand{\Q}{\mathbb Q}
\newcommand{\D}{\mathbb D}
\newcommand{\OO}{\mathcal O}
\newcommand{\Ocal}{\mathcal O}

\newcommand{\Lie}{\mathcal L}
\newcommand{\ii}{\iota}

\newcommand{\T}{\mathcal T}
\newcommand{\Per}{\operatorname{Per}}
\newcommand{\Aut}{\operatorname{Aut}}

\DeclareMathOperator{\Sing}{Sing}

\title[Holomorphic Tischler fibration theorem]{Holomorphic  Tischler fibration theorem for Stein spaces}
\author{Bruno Sc\'ardua}
\address{Instituto de Matem\'atica, Universidade Federal do Rio de Janeiro, Rio de Janeiro, Brazil}
\email{bruno.scardua@gmail.com}
\date{}
\begin{document}
\begin{abstract}
We study a holomorphic analogue of Tischler approximation for closed differential forms normalized by a prescribed holomorphic vector field. The main setting consists of connected normal Stein spaces satisfying a degree-one reflexive Poincar\'e property, with finitely generated first integral homology and a holomorphic $\mathbb C^*$-action admitting a closed normalized reflexive one-form. Under these hypotheses, the given form can be approximated on compact subsets of the regular locus by closed normalized reflexive forms with discrete period groups. When the first Betti number is at least two, the approximating forms may be chosen with rank-two lattice periods. We relate these approximations to holomorphic locally trivial fibrations over quotients of $\mathbb C$, with connected Stein analytic fibers that may be singular. We also establish a finite-equivariant version for suitable Stein quotients and discuss the associated normalization obstructions, including those arising from local singularity topology. Further results concern closed holomorphic time forms, transverse foliations, and the relation between the period groups and the dynamics of complete holomorphic flows. Examples illustrate both the applicability and the limitations of the hypotheses.
\end{abstract}
\maketitle
\tableofcontents
\section{Introduction}\label{sec:intro}
The theory of foliations connects several areas of mathematics, including differential topology, algebraic topology, and dynamical systems. A classical example of this interaction is Tischler's fibration theorem \cite{Tischler1970}, which relates three objects: (i) nowhere-vanishing closed real one-forms on smooth compact manifolds; (ii) fibrations over the circle; and (iii) smooth codimension-one foliations invariant under transverse real flows (see \cite{Plante1972,BehagueScardua2008}). This viewpoint is also useful in the study of certain Anosov flows \cite{Plante1972,Scardua2003}. The holomorphic setting has been considered in \cite{Scardua2003}.

In this work we investigate the noncompact holomorphic setting, principally under the assumption that the ambient space is Stein. Our first objective is to approximate closed holomorphic one-forms on noncompact Stein spaces by closed holomorphic one-forms with discrete periods.
We therefore study closed holomorphic time forms normalized by a fixed complete vector field, especially the infinitesimal generator of a holomorphic $\C^*$-action. Our main results establish normalized approximation on normal Stein spaces under a local reflexive Poincar\'e hypothesis and identify the holomorphic fibrations associated with discrete periods. We also consider finite quotients, surface-specific existence and uniformization, and obstructions arising in singular and nonperiodic settings. The transverse-flow and suspension viewpoint builds on earlier work on invariant holomorphic foliations \cite{Scardua2003,BehagueScardua2008}.

\subsection{Basic notions}
Let $M$ be a connected complex manifold, and let $\Omega_M^1$ denote its sheaf of holomorphic one-forms.
For every open subset \(U\subseteq M\), we have 
\(
{\Omega_M^1(U)=H^0(U,T^*M),}
\)
where \(T^*M\) is the holomorphic cotangent bundle. Given  a nowhere-vanishing holomorphic vector field  $X$ on $M$ we put:
\begin{equation}\label{eq:intro-A}
\begin{aligned}
 \mathcal T_X&=\{\omega\in H^0(M,\Omega_M^1):d\omega=0,\ \omega(X)=1\},\\
 \mathcal F_X&=\{\beta\in H^0(M,\Omega_M^1):d\beta=0,\ \beta(X)=0\}.
\end{aligned}
\end{equation}
We shall refer to $\mathcal T_X$ as the space  of {\it closed time forms} of $X$ and 
to $\mathcal F_X$ as the space   of {\it closed tangent forms} of $X$. 
Whenever $\omega_0 \in \mathcal T_X\ne\varnothing$ we have the structure of an  affine space $\mathcal T_X=\omega_0+\mathcal F_X$. Every $\omega\in\mathcal T_X$ defines a nonsingular holomorphic codimension-one foliation $\mathcal F_\omega=\ker\omega$, invariant under the local flow of $X$  (\cite{BehagueScardua2008}, \cite{Scardua2003}). Our {\it approximation problem} asks {\em whether a closed time form $\omega$ can be approximated on compact sets by $\omega_j\in\mathcal T_X$ whose period images in $\C$ are discrete}. 
Throughout this work we shall adopt the following terminology. 
Given  a holomorphic action $\Phi:\C^*\times M\to M$, it defines a complete holomorphic vector field, called  \emph{infinitesimal generator},  by
\[
 X(p)=\left.\frac{d}{dt}\right|_{t=0}\Phi(e^t,p).
\]
Its  flow is given by $\varphi_t(p)=\Phi(e^t,p)$, and
$\varphi_{2\pi i}=\operatorname{id}_M$.   The zeros of the infinitesimal generator are the fixed points of the action.

For a class $\xi\in H^1(M,\C)$, its \emph{period homomorphism} is
\[
 \operatorname{Per}_{\xi}:H_1(M,\Z)\longrightarrow\C,
 \qquad [\gamma]\longmapsto\langle\xi,[\gamma]\rangle.
\]
The {\it discrete-period locus} is defined as 
\begin{equation}\label{eq:discrete-period-locus}
 \mathcal D(M)=\bigl\{\xi\in H^1(M,\C):
 \operatorname{Per}_{\xi}(H_1(M,\Z))\text{ is discrete in }\C\bigr\}.
\end{equation}
Since every discrete additive subgroup of \(\mathbb C\) is free abelian of rank at most two, every period group associated with a class in \(\mathcal D(M)\) has rank at most two.

We turn to normal complex analytic spaces, using standard terminology as in \cite{GrauertRemmert1984}.

\begin{definition}[Normalized reflexive holomorphic one-form]
Let $S$ be a normal complex analytic space, and let
$S_{\mathrm{reg}}$ denote its regular locus. Write
\[
j:S_{\mathrm{reg}}\hookrightarrow S
\]
for the natural inclusion.

We denote by $\Omega^1_{S_{\mathrm{reg}}}$ the sheaf
of holomorphic differential one-forms on the complex
manifold $S_{\mathrm{reg}}$, and define the sheaf of {\it reflexive} one-forms in $S$ as 
\[
\Omega_S^{[1]}:=j_*\Omega^1_{S_{\mathrm{reg}}}.
\]

Let $X$ be a holomorphic vector field on $S$.
A reflexive holomorphic one-form
\[
\omega\in H^0(S,\Omega_S^{[1]})
\]
is said to be \emph{normalized with respect to $X$}
if
\[
\omega(X)=1 \quad\text{on } S_{\mathrm{reg}}.
\]

It is called a \emph{closed normalized reflexive
holomorphic one-form} if, in addition,
\[
d\omega=0 \quad\text{on } S_{\mathrm{reg}}.
\]
\end{definition}

Recall that, since $S$ is normal, there is a
canonical identification
\[
\Omega_S^{[1]}
\simeq (\Omega_S^1)^{**},
\]
where $\Omega_S^1$ denotes the sheaf of
K\"ahler holomorphic differentials and
$(\cdot)^{**}$ denotes the double dual
as an $\mathcal O_S$-module.

Thus, a reflexive holomorphic one-form
is a holomorphic one-form on $S_{\mathrm{reg}}$,
regarded as a section of its reflexive extension
to $S$. Such a form need not extend as a
K\"ahler differential across the singular locus.
When $S$ is smooth, one has
$\Omega_S^{[1]}=\Omega_S^1$.

\begin{definition}[Degree-one reflexive Poincar\'e property]
\label{def:reflexive-poincare}
Let $S$ be a normal complex analytic space, and write
$j:S_{\rm reg}\hookrightarrow S$ for the inclusion of its regular locus.
Set
\[
\mathcal Z_S^1:=\ker\bigl(d:\Omega_S^{[1]}\longrightarrow
j_*\Omega^2_{S_{\rm reg}}\bigr).
\]
The space $S$ has the \emph{degree-one reflexive Poincar\'e property}
if the sequence of sheaves of abelian groups
\begin{equation}\label{eq:reflexive-poincare}
0\longrightarrow\underline{\C}_S\longrightarrow\mathcal O_S
\xrightarrow{d}\mathcal Z_S^1\longrightarrow0
\end{equation}
is exact. Equivalently, every closed reflexive holomorphic one-form
has a local holomorphic primitive, including at singular points. 

\end{definition}

The degree-one reflexive Poincar\'e property  is not automatic for a normal Stein space. 
It holds
for several natural classes of complex spaces.
It is automatic for smooth complex manifolds by
the classical holomorphic Poincar\'e lemma.
It also holds for normal spaces with isolated
complete-intersection singularities of complex
dimension at least three, as well as their
products with complex manifolds
(Proposition~\ref{prop:isolated-ici-poincare}).

Another class is provided by finite quotients
of smooth complex manifolds. Indeed, closed
reflexive one-forms pull back locally to closed
holomorphic one-forms on the covering manifold;
their local primitives can be averaged over the
finite group and descended to the quotient.
Thus the property holds, for instance, for
the quadratic cone
\[
\{xy=z^2\}\subset\mathbb C^3
\]
and, by Proposition~\ref{prop:isolated-ici-poincare}, for the three-dimensional ordinary double
point
\[
\{xy-zw=0\}\subset\mathbb C^4.
\]

\subsection{Tischler approximation  on normal Stein spaces}
Our first main result extends normalized Tischler approximation to normal Stein spaces.

\begin{mainresult}[Tischler approximation on normal Stein spaces: analytic form]\label{thm:intro:normal}
Let $S$ be a connected normal Stein space equipped with a  holomorphic $\C^*$-action with infinitesimal generator $X$.
Assume that:
\begin{enumerate}
\item The first homology 
$H_1(S,\Z)$ is  finitely generated.
\item $S$ has  the degree-one reflexive Poincar\'e property.
\end{enumerate}
Given a  closed reflexive one-form
$\omega$ normalized by  $X$, then $\omega$ is a compact-open limit on $S_{\rm reg}$ of closed,
normalized reflexive forms with discrete period groups. If $b_1(S)\geq 2$,
the approximating forms can be chosen so that
their period groups are lattices of rank two
in $\mathbb C$.
\end{mainresult}

In complex dimension two, the existence of a closed
normalized holomorphic one-form is automatic for
fixed-point-free holomorphic $\mathbb C^*$-actions
on normal Stein surfaces, as shown in
Proposition~\ref{torus:prop:surface}.

We can also state it in the geometric form below:

\begin{mainresult}[Tischler approximation on normal Stein spaces: geometric form]
\label{thm:geometric-reflexive-tischler}
Let $S$ be a connected normal Stein space with
finitely generated $H_1(S,\mathbb Z)$, satisfying
the degree-one reflexive Poincar\'e property.
Let $X$ be the infinitesimal generator of a
holomorphic $\mathbb C^*$-action on $S$ admitting  a closed reflexive normalized 
holomorphic one-form $\omega$.

Then the holomorphic foliation
defined by $\omega$ on $S_{\mathrm{reg}}$
is approximated,  in the compact-open
topology, by  fibrations \[
F_j:S\longrightarrow\mathbb C/\Lambda_j
\]
where 
$\Lambda_j\subset\mathbb C, j \in \mathbb N$ is a  discrete subgroup. 
\end{mainresult}

A few remarks about the geometric version is:

$\bullet$ We have 
\(
F_j^*(dt_j)\longrightarrow\omega
\)
uniformly on compact subsets of $S_{\mathrm{reg}}$.
Here $dt_j$ denotes the invariant holomorphic
one-form on $\mathbb C/\Lambda_j$ induced by
$dz$ on $\mathbb C$.

$\bullet$ Each base $\mathbb C/\Lambda_j$ is biholomorphic
either to $\mathbb C^*$ or to an elliptic curve.

$\bullet$ If $b_1(S)\geq2$, the subgroups
$\Lambda_j$ may all be chosen to be lattices
of rank two, so that every base
$\mathbb C/\Lambda_j$ is an elliptic curve.

\begin{mainresult}[Finite-quotient Tischler approximation]\label{thm:intro:finite}
Let $S=M/G$ be a finite quotient of a connected Stein manifold $M$
with finitely generated first homology, and assume that $M$ carries a holomorphic $\C^*$-action whose
infinitesimal generator is nowhere zero and which commutes with $G$.
Consequently $S$ inherits a fixed-point-free holomorphic $\C^*$-action.
Every $G$-invariant normalized closed holomorphic form on $M$ admits
$G$-invariant normalized approximants with discrete periods; the
descended reflexive forms on $S$ have discrete periods on $S_{\rm reg}$.

\end{mainresult}

\subsection{Manifold case}
When $S$ is smooth, the reflexive Poincar\'e property follows from the classical holomorphic Poincar\'e lemma (see \cite{Wells2008}).
The general approximation theorem specializes to the following smooth statement.

\begin{maincorollary}[Tischler approximation: manifold case]
\label{thm:intro:a}
Let $M$ be a connected Stein manifold with finitely generated first homology group 
$H_1(M,\Z)$. Let $X$ be the nowhere-vanishing infinitesimal
generator of a holomorphic $\C^*$-action on $M$. 
Then the following three cases occur:
\begin{enumerate}[label=\textup{(\roman*)},leftmargin=*]
\item If $b_1(M)=0$, then $\mathcal T_X=\varnothing$.
\item If $b_1(M)=1$ and $\mathcal T_X\ne\varnothing$, every
$\omega\in\mathcal T_X$ has a nonzero discrete cyclic period group;
its integral defines a holomorphic locally trivial fibration with
connected fibers over $\C^*$, whose fibers are the leaves of
$\ker\omega$. 
\item If $b_1(M)\ge2$ and $\mathcal T_X\ne\varnothing$, every
$\omega\in\mathcal T_X$ is a compact-open limit of closed normalized
holomorphic forms with rank-two lattice period groups. Their leaves
are connected fibers of holomorphic locally trivial fibrations over
elliptic curves.
\end{enumerate}
In each case where $\mathcal T_X\ne\varnothing$, the cohomology
classes of its members constitute the affine hyperplane
\[
\bigl\{\xi\in H^1(M,\C):\langle\xi,[\gamma]\rangle=2\pi i\bigr\},
\]
where $[\gamma]$ is the common orbit-circle class.
\end{maincorollary}

 On Stein surfaces, fixed-point-free multiplicative actions also yield an automatic existence result.

\begin{mainresult}[Uniformization theorem]
\label{thm:intro:b}
Let $M$ be a connected Stein surface carrying a fixed-point-free holomorphic \(\mathbb C^*\)-action with infinitesimal generator \(X\).
Then:
\begin{enumerate}[label=\textup{(\roman*)},leftmargin=*]
\item  There is a closed holomorphic
one-form $\omega$ on $M$ satisfying $\omega(X)=1$. In particular
$H^1(M,\C)\ne0$ and $b_1(M)\ge1$.
\item  The universal cover $\widetilde M$ is
biholomorphic to $\C^2$ or $\C\times\mathbb D$. In suitable
global coordinates $\widetilde M\simeq\C\times\Sigma$,
$\Sigma\in\{\C,\mathbb D\}$, the lifted vector field is
$\widetilde X=\partial/\partial z$, with flow
$\widetilde\varphi_t(z,w)=(z+t,w)$.
\end{enumerate}

If, moreover  $H_1(M,\Z)$ is finitely generated,
then there is a holomorphic locally trivial fibration $F:M\to B$
with connected Stein curve fibers transverse to $X$ as follows:
\begin{enumerate}
\item $B$ is a cylinder,  $B\simeq\C^*$ when $b_1(M)=1$.
\item $B$ is a torus,  $B\simeq\C/\Lambda$ for a rank-two lattice $\Lambda\subset\C$
when $b_1(M)\ge2$. 
\end{enumerate}
The fibers are the leaves of the kernel
of a closed holomorphic one-form $\omega$  normalized by $X$.

\end{mainresult}

The universal-cover conclusion in Theorem~\ref{thm:intro:b} also holds for complete holomorphic flows carrying closed normalized time forms (Theorem~\ref{thm:uniform-time}). The approximation conclusion for multiplicative actions, however, differs from the situation for additive actions, as the counterexamples and cohomological discussion below illustrate.

The periodicity assumption in our approximation results
is essential. Indeed, we construct Stein surfaces
carrying free and proper holomorphic $\mathbb C$-actions
which admit closed normalized time forms, but for
which no such form has discrete periods. This construction is established in Theorem~\ref{stein:thm:counterexample}.

We also consider holomorphic $\mathbb C^*$-actions
with singularities. In this setting, we construct
normalized logarithmic time forms associated with
holomorphic semi-invariants and obtain
discrete-period approximation on suitable
invariant Stein open subsets (Theorem~\ref{thm:intro:e}).

\subsection{Foliations invariant under transverse flows}

The hypotheses of the following theorem admit a natural
interpretation in holomorphic foliation theory. Indeed,
a nonsingular holomorphic foliation $\mathcal F$ of
codimension one on a complex manifold $M$, invariant
under a complete holomorphic flow everywhere transverse
to its leaves, is defined by a closed holomorphic
one-form $\omega$ normalized by the infinitesimal
generator $X$ of the flow, that is,
\[
\mathcal F=\ker\omega,\qquad
d\omega=0,\qquad \omega(X)=1.
\]
Conversely, any such pair $(\omega,X)$ determines
a nonsingular holomorphic foliation invariant under
the complete transverse flow generated by $X$.
This correspondence belongs to the classical framework
of invariant holomorphic foliations developed in
\cite{Scardua2003}; see also
\cite{BehagueScardua2008}.

\begin{theorem}[Closures of invariant leaves]\label{thm:intro:f}
Let $M$ be a connected complex manifold equipped with a codimension one holomorphic foliation \(\mathcal F\) which is invariant under a transverse holomorphic flow $\Phi:\C\times M\to M$.
Then we have the following possibilities for the leaves of $\mathcal F$: 

\begin{enumerate}
\item All leaves are closed;
\item Their closures are proper real-analytic Levi-flat hypersurfaces; 
\item All leaves are dense. 
\end{enumerate}
\end{theorem}

More precisely, let $X$ be the infinitesimal generator of the flow, and let $\omega$ be a closed holomorphic one-form on $M$ defining $\mathcal F$ and satisfying
\[
\omega(X)=1.
\]
Set $H=\overline{\operatorname{Per}(\omega)}\subset\C$, the closure of the period group of $\omega$. 
The three cases correspond, respectively, to $H$ discrete, $H$ proper and nondiscrete, and $H=\C$.
Examples~\ref{ex:intermediate} and~\ref{ex:dense} realize the latter two cases on Stein surfaces.

\noindent\emph{Foliation-theoretic meaning.}
Theorem~\ref{thm:intro:a} approximates the nonsingular invariant
foliation $\mathcal F_\omega=\ker\omega$, through foliations transverse
to and invariant under the same complete flow, by holomorphic
fiber-bundle foliations over $\C^*$ or an elliptic curve. The
approximation is uniform on compact subsets at the level of their
defining one-forms and retains $\omega_j(X)=1$ exactly. For
$b_1(M)\ge2$ the approximating bases can be taken to be elliptic
curves. This is a holomorphic, noncompact counterpart to Tischler's
approximation of real closed nonsingular one-forms, with the additional
requirement that a prescribed transverse flow remain fixed. The
geometric reason that discrete periods yield closed fibers is made
precise in Theorem~\ref{thm:global-stability} below.

\subsection{Application to jointly-integrable holomorphic hyperbolic flows}
We now turn to the relation between hyperbolic dynamics
and the period group of a normalized closed holomorphic
one-form, under a joint integrability hypothesis.

\begin{mainresult}[Period rigidity for jointly integrable hyperbolic flows]
\label{thm:intro:hyperbolic}
Let $M$ be a connected complex manifold, and let $X$ be a
nowhere-vanishing complete holomorphic vector field with flow
$\Phi:\C\times M\to M$. Assume an invariant holomorphic splitting
\[
TM=E^s\oplus\C X\oplus E^u
\]
with $E^s,E^u$ of positive rank, with $E^s\oplus E^u$
involutive, and a Hermitian metric $h$ and constants $C,\lambda>0$
for which, for all real $t\ge0$,
\[
\|D\Phi_t(v_s)\|_h\le Ce^{-\lambda t}\|v_s\|_h,
\qquad
\|D\Phi_{-t}(v_u)\|_h\le Ce^{-\lambda t}\|v_u\|_h.
\]
Let $\omega$ be the unique normalized holomorphic one-form
with $\ker\omega=E^s\oplus E^u$ (necessarily closed).
Suppose $\Phi_T(p)=p$ for some $p\in M$ and real $T>0$.
If the global dynamical stable and unstable
manifolds of $P=\Phi_T|_{L_p}$ on the joint leaf $L_p$
intersect in exactly one point for every pair of base points,
then
\[
\operatorname{Per}(\omega)=\{t\in\C:\Phi_t(p)=p\}.
\]
This group is nonzero and discrete, and $\ker\omega$
defines a holomorphic fiber-bundle foliation over
$\C/\operatorname{Per}(\omega)$.
When $M$ is Stein, this group is cyclic, so the base is $\C^*$.
\end{mainresult}

The jointly integrable framework is developed in the body of the paper, where the dynamical product hypothesis yields the stated period rigidity.

\subsubsection{Some references on former works}
In addition to \cite{Scardua2003,BehagueScardua2008}, related contributions
to the existence of holomorphic submersions and closed holomorphic
one-forms on Stein manifolds were obtained by Forstneri\v{c}
\cite{Forstneric2004} and Majcen \cite{Majcen2007}.
Our approach builds on these ideas, with particular emphasis on
the approximation of closed holomorphic one-forms normalized
by a prescribed complete vector field.

We extend this approximation problem to normal Stein spaces
using reflexive holomorphic differentials and investigate the
holomorphic fibrations arising from discrete period groups.
We also study the obstructions to normalization near singularities,
drawing on the classical theory of holomorphic
$\mathbb C^*$-actions developed by Orlik--Wagreich
\cite{OrlikWagreich1971} and continued in
\cite{CamachoMovasatiScardua2009}.
\section{Reflexive holomorphic Tischler approximation on normal Stein spaces}
\label{sec:reflexive-tischler}
We begin with the reflexive de Rham description of degree-one cohomology and then establish normalized approximation and its fibration consequence.
Throughout this section $S$ is a connected normal Stein complex space.
We use the sheaves $\Omega_S^{[1]}$ and $\mathcal Z_S^1$ and the
degree-one reflexive Poincar\'e property from
Definition~\ref{def:reflexive-poincare}, in particular the exact
sequence~\eqref{eq:reflexive-poincare}.

\begin{lemma}[Regular-locus representatives of loops]
\label{lem:regular-loops}
Let $S$ be a connected normal complex space. For a base point in
$S_{\rm reg}$, the inclusion induces a surjection
\[
\pi_1(S_{\rm reg})\longrightarrow\pi_1(S),
\]
and hence a surjection on first integral homology. In particular,
the period character defined by local holomorphic primitives on $S$
has the same image when computed on $S_{\rm reg}$.
\end{lemma}
\begin{proof}
Complex spaces are locally contractible. By local conical structure
and normality, each point has arbitrarily small contractible
neighborhoods whose regular loci are connected and path connected.
Subdivide a loop into finitely many arcs in such neighborhoods.
Move the subdivision points into the regular locus in the overlaps,
keeping the base point fixed, and join the resulting endpoints by
paths in the regular part of each neighborhood. Contractibility
shows that these replacements give a homotopic loop in $S$.
This proves surjectivity on fundamental groups and on their
abelianizations. The last assertion follows by restriction of the
period character. Local contractibility also supplies the covering
spaces used below.
\end{proof}

\begin{proposition}[Reflexive de Rham]
\label{prop:reflexive-de-rham}
If a connected normal Stein space $S$ satisfies \eqref{eq:reflexive-poincare},
then taking periods induces a natural isomorphism
\[
H^0(S,\mathcal Z_S^1)/d\mathcal O(S)\ \simeq\ H^1(S,\C).
\]
In particular, periods of global closed reflexive one-forms along loops of
$S_{\rm reg}$ depend only on their images in $H_1(S,\Z)$.
\end{proposition}
\begin{proof}
The cohomology sequence of \eqref{eq:reflexive-poincare} yields
\[
H^0(S,\mathcal O_S)\xrightarrow{d}H^0(S,\mathcal Z_S^1)
\longrightarrow H^1(S,\underline{\C}_S)
\longrightarrow H^1(S,\mathcal O_S).
\]
The last term vanishes by Cartan's Theorem B. The connecting map is the
usual period class: local primitives differ by locally constant functions.
This gives the isomorphism and the assertion about loops.
\end{proof}

The degree-one reflexive de Rham description makes it possible to approximate normalized classes without changing the prescribed vector field.

\begin{theorem}[Reflexive Tischler approximation]
\label{thm:normal-reflexive-approx}
Let $S$ be a connected normal Stein space satisfying the degree-one
reflexive Poincar\'e property, and assume $H_1(S,\Z)$ is finitely generated.
Let $X$ generate a holomorphic $\C^*$-action, with complex-time flow
$\varphi_{2\pi i}=\mathrm{id}$. Suppose a global closed reflexive one-form
$\omega\in H^0(S,\mathcal Z_S^1)$ satisfies $\omega(X)=1$ on $S_{\rm reg}$.
Then there exist $\omega_j\in H^0(S,\mathcal Z_S^1)$ such that
\[
\omega_j(X)=1,\qquad
\omega_j\longrightarrow\omega\quad\text{uniformly on compact subsets of }S_{\rm reg},
\]
and the period group of $\omega_j|_{S_{\rm reg}}$ is discrete in $\C$.
When $b_1(S)\ge2$, the approximants may be chosen with rank-two lattice
periods. By Corollary~\ref{cor:reflexive-fibration}, they define holomorphic locally trivial fibrations of $S$ with connected Stein analytic fibers, possibly singular.
\end{theorem}
\begin{proof}
Let $\gamma$ be the homology class in $H_1(S,\Z)$ of any circle orbit
in $S_{\rm reg}$. This locus is invariant under the action. Normalization
implies $\int_\gamma\omega=2\pi i$, so $\gamma$ is non-torsion. Since
$H_1(S,\Z)$ is finitely generated, the normalization condition
selects the affine hyperplane
\[
\mathcal H_\gamma=\{\xi\in H^1(S,\C):
\langle\xi,\gamma\rangle=2\pi i\}.
\]
This hyperplane contains
a dense set of classes whose periods belong to a common discrete
Gaussian rational lattice (which may depend on the class). Indeed,
one chooses a rational basis for the free part of integral homology,
solves the single rational affine constraint, and approximates the remaining
real and imaginary coordinates by rationals. For $b_1(S)\ge2$, one may
ensure that the period image contains two real-linearly independent elements by a small choice of a
nonreal free coordinate.

Choose classes $\xi_j\to[\omega]$ in that rational locus. By
Proposition~\ref{prop:reflexive-de-rham}, represent a complex basis of
$H^1(S,\C)$ by global closed reflexive forms. Small linear combinations
of these representatives give closed forms $\alpha_j\to\omega$ on compact
subsets of the smooth locus and with $[\alpha_j]=\xi_j$.

Average $\alpha_j$ over the circle subgroup of the given $\C^*$-action.
The resulting $\omega_j$ are closed reflexive holomorphic forms, retain
classes $\xi_j$ because the circle action is homotopic to the identity,
and converge to $\omega$: the latter is invariant by Cartan's formula
since $d\omega=0$ and $\omega(X)=1$ on the smooth locus.
The averages are invariant, so Cartan's formula gives
$d(\omega_j(X))=0$ on $S_{\rm reg}$; this locus is connected by normality.
The orbit integral then gives
$2\pi i\,\omega_j(X)=\langle\xi_j,\gamma\rangle=2\pi i$.
Thus the normalization is exact. The selected cohomology classes have
discrete period groups, and local exactness together with
Lemma~\ref{lem:regular-loops} identifies the periods on
$S_{\rm reg}$ with those on $S$.
\end{proof}

Discrete periods have a direct geometric interpretation, which we now formulate for normal spaces.

\begin{corollary}[Reflexive Tischler fibration]
\label{cor:reflexive-fibration}
Let $S$ be a connected normal Stein space satisfying the degree-one
reflexive Poincar\'e property, and let $X$ be the infinitesimal
generator of a holomorphic $\C^*$-action on $S$.
Suppose that $\omega\in H^0(S,\mathcal Z_S^1)$ is a closed
reflexive holomorphic one-form with $\omega(X)=1$ on $S_{\rm reg}$
and discrete period group $\Lambda=\Per(\omega)\subset\C$.
Then integration induces a surjective holomorphic locally trivial
fiber bundle, in the category of complex analytic spaces,
\[
F:S\longrightarrow\C/\Lambda,
\]
whose fibers are connected Stein analytic spaces, possibly singular.
The identity $F^*(dt)=\omega$ holds on $S_{\rm reg}$, and the
holomorphic flow satisfies
\[
F(\varphi_t(p))=F(p)+[t]\quad(t\in\C).
\]
Moreover, $2\pi i\in\Lambda$, so $\C/\Lambda$ is biholomorphic
either to $\C^*$ or to an elliptic curve.
No finite-generation hypothesis on $H_1(S,\Z)$ is required.
\end{corollary}
\begin{proof}
By the degree-one reflexive Poincar\'e property, the form $\omega$
has local holomorphic primitives on $S$; differences of primitives
are locally constant. Their additive monodromy is the period
homomorphism $\pi_1(S)\to\C$ of $\omega$, whose image is $\Lambda$.
The local primitive property ensures that the period homomorphism
on the regular locus factors through $\pi_1(S)$: small loops
around the singular set contribute no additional monodromy.
Equivalently, the local primitives define a class in $H^1(S,\C)$
whose restriction gives the periods on $S_{\rm reg}$.
Thus the local primitives glue modulo $\Lambda$ into a holomorphic
map $F:S\to\C/\Lambda$ satisfying $F^*(dt)=\omega$ on the
regular locus. Here Lemma~\ref{lem:regular-loops} identifies the
period images on $S$ and $S_{\rm reg}$.

If $f$ is a local holomorphic primitive of $\omega$, then
$X(f)=1$ on the regular locus and hence throughout its domain.
The local holomorphic flow therefore gives product charts transverse
to the analytic level sets of $f$, including at singular points.

Let $\varphi_t$ be the complete holomorphic flow. Since
$\omega(X)=1$ on the regular locus and both sides extend
holomorphically, the time derivative of $F(\varphi_t(p))$ is the
invariant unit tangent vector on $\C/\Lambda$. Hence
$F(\varphi_t(p))=F(p)+[t]$ throughout $S$, which also gives
surjectivity.

Form the pullback of the covering $\C\to\C/\Lambda$:
\[
\widehat S=S\times_{\C/\Lambda}\C,
\qquad T:\widehat S\to\C.
\]
This covering is connected, since its monodromy surjects onto
$\Lambda$ by the definition of the period group. The lifted
complete flow is given by
$\widehat\varphi_t(p,z)=(\varphi_t(p),z+t)$, and satisfies
$T(\widehat\varphi_t(q))=T(q)+t$.
Setting $L=T^{-1}(0)$, we obtain mutually inverse holomorphic
maps
\[
L\times\C\longrightarrow\widehat S,\qquad
(q,t)\longmapsto\widehat\varphi_t(q),
\]
and
\[
\widehat S\longrightarrow L\times\C,\qquad
q\longmapsto\bigl(\widehat\varphi_{-T(q)}(q),T(q)\bigr).
\]
Thus $\widehat S\simeq L\times\C$ and $L$ is connected.
Over a simply covered open disc in $\C/\Lambda$, this product
structure descends to a holomorphic trivialization of $F$ with
fiber $L$. The fiber is a closed analytic subspace of the Stein
space $S$ and hence Stein. Finally, the orbit circle in the
regular locus has period $2\pi i$, so $2\pi i\in\Lambda$.
A discrete subgroup of $\C$ containing $2\pi i$ has rank one
or two, with quotient respectively $\C^*$ or an elliptic curve.
\end{proof}

\begin{proof}[Proof of Theorem~\ref{thm:intro:normal}]
Under the hypotheses of Theorem~\ref{thm:intro:normal},
Theorem~\ref{thm:normal-reflexive-approx} provides closed,
normalized reflexive forms $\omega_j\to\omega$ on compact subsets
of $S_{\rm reg}$, with discrete periods and, when $b_1(S)\ge2$,
with rank-two lattice periods. This proves every approximation
assertion of the introductory theorem. For completeness, its
geometric consequence follows by applying
Corollary~\ref{cor:reflexive-fibration} to each $\omega_j$:
we obtain a holomorphic locally trivial fibration
$S\to\C/\Per(\omega_j)$ with connected Stein analytic fibers.
Since $2\pi i\in\Per(\omega_j)$, the base is $\C^*$ or an
elliptic curve, according to the rank of the period group.
\end{proof}

For the geometric form we have:

\begin{proof}[Proof of Theorem~\ref{thm:geometric-reflexive-tischler}]
By Theorem~A, there exists a sequence of closed
normalized reflexive holomorphic one-forms
$\omega_j$ with discrete period groups
\[
\Lambda_j=\operatorname{Per}(\omega_j)
\]
such that $\omega_j\to\omega$ uniformly on
compact subsets of $S_{\mathrm{reg}}$.

By Corollary~\ref{cor:reflexive-fibration},
integration of each $\omega_j$ induces a
holomorphic locally trivial fibration
\[
F_j:S\longrightarrow\mathbb C/\Lambda_j
\]
with connected Stein analytic fibers,
satisfying
\[
F_j^*(dt_j)=\omega_j
\quad\text{on }S_{\mathrm{reg}}.
\]
The required convergence follows immediately.

Since $\omega_j(X)=1$, integration along
an orbit circle gives
\[
2\pi i\mathbb Z\subseteq\Lambda_j.
\]
Thus each discrete group $\Lambda_j$ has
rank one or two. If $b_1(S)\geq2$, Theorem~A
allows the approximating forms to be chosen
with rank-two lattice periods, yielding
elliptic curves as bases.

Finally, the foliations induced by the
fibrations are defined on $S_{\mathrm{reg}}$
by $\ker\omega_j$, and the convergence
$\omega_j\to\omega$ gives the asserted
approximation of their normalized defining
forms.
\end{proof}

The local primitive property is available beyond quotient singularities; the next proposition supplies a useful class of examples.

\begin{proposition}[A non-quotient class with local primitives]
\label{prop:isolated-ici-poincare}
A normal complex space with isolated complete-intersection singularities
of complex dimension at least three satisfies the degree-one reflexive
Poincar\'e property. The same is true for its product with a complex
manifold.
\end{proposition}
\begin{proof}
The required local topological vanishing for isolated complete intersections
is recorded in \cite[Proposition~3.3]{Jorder2014}. We recall the elementary
argument. At a smooth point this is the holomorphic Poincar\'e lemma. At an isolated
complete-intersection singularity of dimension $n\ge3$, a sufficiently
small punctured representative is simply connected by Hamm's
connectivity theorem for isolated complete intersections
\cite{Hamm1971} (for hypersurfaces, see Milnor\cite{Milnor1968}). A closed holomorphic one-form there has a
single-valued holomorphic primitive by path integration. Normality
extends the primitive across the isolated point, and its differential
agrees with the original reflexive form on the smooth locus.
For a product with a complex manifold, integrate first on the punctured
singular factor and on small coordinate polydiscs in the smooth factor;
the resulting local primitive extends across the product singular locus
by normality. Equivalently, use the local product decomposition of the
punctured neighborhoods and the ordinary Poincar\'e lemma in the smooth
directions.
\end{proof}
The following example illustrates the approximation theorem
on a singular Stein space that is not locally a quotient
of a smooth complex manifold by a finite group.
\begin{example}[Non-quotient singular Stein approximation]
\label{ex:nonquotient-nodal-product}
Let $Y=\{xy-zw=0\}\subset\C^4$, the three-dimensional ordinary double
point, and let $S=Y\times(\C^*)^2$. It is a normal affine Stein space
with singularities and satisfies the degree-one reflexive Poincar\'e
property by Proposition~\ref{prop:isolated-ici-poincare}. The local
class group of the threefold node is infinite cyclic, so its singular
germ is not a finite smooth quotient germ (which is $\Q$-factorial).
With coordinates $(u,v)$ on the torus factor, set
\[
X=u\partial_u,\qquad
\omega_\lambda=\frac{du}{u}+\lambda\frac{dv}{v}.
\]
If $\lambda\in\R\setminus\Q$, then
$\Per(\omega_\lambda)=2\pi i(\Z+\lambda\Z)$ is nondiscrete.
For $\lambda_j\in\Q+i\Q$ with nonzero imaginary parts tending to
$\lambda$, the $\omega_{\lambda_j}$ are closed and normalized with
rank-two lattice periods and converge to $\omega_\lambda$.
This gives an explicit case beyond finite smooth quotients.
\end{example}

\section{Equivariant approximation on finite Stein quotients}
Finite quotients provide a complementary setting in which approximation is carried out equivariantly on a smooth Stein covering space before descent.

\label{sec:finite-quotient}
The preceding local obstructions concern the \emph{existence} of normalized forms. There is nevertheless a substantial class of genuinely singular normal Stein spaces where the approximation procedure itself survives. We give a finite-equivariant version of the smooth theorem, and explicitly distinguish descent of closed forms and periods from the stronger assertion that their leaves form bundles on the singular quotient.

\begin{lemma}
\label{lem:reflection-one-form-descent}
Let $V$ be a complex vector space, let $K\subset\mathrm{GL}(V)$ be
finite, and suppose that $V/K$ is smooth near the image of the origin.
Then every $K$-invariant holomorphic one-form defined near the origin
is the pullback of a unique holomorphic one-form on the corresponding
quotient germ. Moreover, for a sequence of such invariant forms,
compact-open convergence upstairs implies compact-open convergence
of their descended forms downstairs.
\end{lemma}
\begin{proof}
Since the quotient $V/K$ is smooth near the image of
the origin, the Chevalley--Shephard--Todd theorem
implies that $K$ is generated by complex
pseudoreflections, that is, elements whose fixed-point
subspaces are hyperplanes
(see \cite{ChevalleyShephardTodd}).

Let $D$ denote the union of the reflecting
hyperplanes of $K$. Away from the image of $D$,
the quotient map $\pi:V\to V/K$ is unramified,
and every $K$-invariant holomorphic one-form
descends uniquely to a holomorphic one-form
on the corresponding open subset of $V/K$.

We now examine the extension across the branch
divisor. At a generic point of a component of
$D$, the stabilizer is cyclic, and suitable
holomorphic coordinates may be chosen so that
its action is given by
\[
(z_1,z')\longmapsto(\zeta z_1,z'),
\qquad \zeta^m=1.
\]
The corresponding quotient coordinates are
\[
w_1=z_1^m,\qquad w'=z'.
\]

Let
\[
\alpha=a_1(z)\,dz_1+
\sum_{j>1}a_j(z)\,dz_j
\]
be a $K$-invariant holomorphic one-form.
Invariance under the cyclic stabilizer implies
that its coefficients have the form
\[
a_1(z)=z_1^{m-1}b_1(z_1^m,z'),
\qquad
a_j(z)=b_j(z_1^m,z')
\quad (j>1),
\]
where the functions $b_j$ are holomorphic.
Consequently,
\[
\alpha
=
\frac{1}{m}b_1(w)\,dw_1+
\sum_{j>1}b_j(w)\,dw_j,
\]
showing that the descended form extends
holomorphically across the generic points
of every branch divisor.

The remaining exceptional subset of the
smooth quotient germ has complex codimension
at least two. Since the quotient is smooth,
the coefficient functions of the descended
form extend holomorphically across this
subset by Hartogs' extension theorem.
Thus $\alpha$ is the pullback of a holomorphic
one-form on the quotient germ. Uniqueness
follows from uniqueness of descent on the
unramified locus, which is dense.

This conclusion is also consistent with
Solomon's description of invariant
differential forms for finite complex
reflection groups \cite{Solomon1963}.

Finally, suppose that $\alpha_\nu$ is a
sequence of $K$-invariant holomorphic
one-forms converging uniformly on compact
subsets upstairs. In the cyclic quotient
coordinates above, the coefficients of
the descended forms depend continuously
on the invariant coefficients of
$\alpha_\nu$, as follows from their
power-series expansions and Cauchy
estimates. Hence the descended forms
converge uniformly on compact subsets
away from the exceptional subset of
codimension at least two.

Near that exceptional subset, choose
small coordinate polydiscs in the smooth
quotient whose distinguished boundaries
avoid the exceptional set. The maximum
modulus principle, applied to the
holomorphically extended coefficient
functions, gives uniform convergence
on compact subsets of these polydiscs.
A finite covering argument then proves
compact-open convergence throughout
the quotient germ.
\end{proof}
\begin{theorem}[Tischler approximation on finite Stein quotients]
\label{thm:finite-quotient-tischler}
Let $M$ be a connected Stein manifold with finitely generated $H_1(M,\Z)$, and let $G$ be a finite group acting holomorphically on $M$. Suppose that a holomorphic $\C^*$-action on $M$ commutes with $G$ and has nowhere-vanishing infinitesimal generator $X$. Given a $G$-invariant closed holomorphic one-form $\omega$ such that $\omega(X)=1$, there are $G$-invariant closed holomorphic one-forms $\omega_j$ on $M$ satisfying
\[
 \omega_j(X)=1,\qquad \omega_j\longrightarrow\omega
 \quad\text{uniformly on compact subsets of }M,
 \qquad \Per(\omega_j)\subset\C\ \text{discrete}.
\]
Writing $S=M/G$ and $\pi:M\to S$ for the quotient map, the forms $\omega,\omega_j$ descend to reflexive one-forms $\bar\omega,\bar\omega_j$ on the normal Stein space $S$. Their restrictions to $S_{\rm reg}$ are closed and normalized by the descended vector field $\bar X$; furthermore $\bar\omega_j\to\bar\omega$ uniformly on compact subsets of $S_{\rm reg}$ and
\[
 \Per(\bar\omega_j|_{S_{\rm reg}})
 \subseteq\tfrac1{|G|}\Per(\omega_j),
\]
so these descended period groups are discrete.
\end{theorem}
\begin{proof}
Let $\gamma$ denote the common orbit-circle class in $H_1(M,\Z)$. Since $\omega(X)=1$, its period is $2\pi i$, so its image in $H_1(M,\Q)$ is nonzero. Because $G$ commutes with the $\C^*$-action, $\gamma$ is $G$-fixed. Consider the rational vector space
\[
 V_\Q=H^1(M,\Q)^G,
 \qquad V_\C=V_\Q\otimes_\Q\C=H^1(M,\C)^G.
\]
The equality follows from finite-group averaging. The affine equation $\langle\xi,\gamma\rangle=2\pi i$ has a rational solution after division by $2\pi i$: the functional $\langle\,,\gamma\rangle$ is rational and nonzero on $V_\Q$. Thus the normalized affine hyperplane within $V_\C$ has dense points with coordinates in $2\pi i(\Q+i\Q)$ relative to a rational basis, while preserving the orbit-period equation exactly. Choose such invariant classes $\xi_j\to[\omega]$. Because $H_1(M,\Z)$ is finitely generated, the image of each $\xi_j$ is contained in a Gaussian rational lattice $N_j^{-1}2\pi i\Z[i]$, for some positive integer $N_j$, and hence is discrete.

By holomorphic de Rham theory on the Stein manifold $M$, choose closed global holomorphic representatives of a finite basis of $H^1(M,\C)$. Perturbing $\omega$ by small linear combinations gives closed forms $\alpha_j\to\omega$ in the compact-open topology with $[\alpha_j]=\xi_j$. Average $\alpha_j$ over $G$ and then over the circle subgroup of $\C^*$. The averaged forms are closed and $G$-invariant, their cohomology classes remain $\xi_j$, and they converge to $\omega$ (which is itself invariant under the circle by Cartan's formula). By the orbit-period argument in Theorem~\ref{stein:thm:hyperplane}, every averaged form has contraction $1$ with $X$. Denote these averaged forms by $\omega_j$.

For descent, replace $G$ by its effective quotient
$G/\ker(G\to\operatorname{Aut}(M))$, if necessary. This does not
change $S=M/G$ or the invariant forms, and its order divides the
order of the original group. With the resulting effective action,
let $M^\circ\subset M$ be the free-action locus and put $S^\circ=M^\circ/G$. The map $\pi:M^\circ\to S^\circ$ is a finite unramified covering, so the invariant forms descend uniquely to $S^\circ$. For descent across smooth branch points, fix $q\in S_{\rm reg}$
and a lift $p$. Finite-group local linearization
\cite{Bochner1945} models the quotient germ at $q$ by the linear
action of the stabilizer $G_p$. Since that germ is smooth, the
Chevalley--Shephard--Todd criterion
\cite{ChevalleyShephardTodd} applies. Lemma~\ref{lem:reflection-one-form-descent}
therefore extends the invariant form uniquely to a holomorphic
one-form near $q$. On the complement of $S_{\rm reg}$, which has
codimension at least two by normality, these forms determine a
unique reflexive differential in $j_*\Omega^1_{S_{\rm reg}}$. Closedness and normalization are equalities on the regular locus because they hold over the dense unramified set. Compact-open convergence of the descended forms on $S_{\rm reg}$
follows by applying the convergence assertion of
Lemma~\ref{lem:reflection-one-form-descent} in finite local quotient
charts.

To compare periods, represent a loop in $S_{\rm reg}$ by a loop avoiding the branch locus, which has real codimension two and can be avoided by a small perturbation; closedness preserves its integral. A lift to $M^\circ$ ends at a point in the same $G$-orbit. If the monodromy element is $g\in G$, its order $m$ divides $|G|$ (Lagrange's theorem). The $m$-fold concatenation lifts to a closed loop, giving $m\int_\gamma\bar\omega_j\in\Per(\omega_j)$. Therefore also $|G|\int_\gamma\bar\omega_j\in\Per(\omega_j)$, which proves the displayed containment and discreteness. This argument makes no assertion that integration defines a locally trivial fibration across $S_{\rm sing}$.
\end{proof}

\begin{proof}[Proof of Theorem~\ref{thm:intro:finite}]
The hypotheses of the introductory theorem are those of
Theorem~\ref{thm:finite-quotient-tischler}. That theorem constructs
$G$-invariant closed normalized approximants upstairs and proves
both their descent to closed normalized reflexive forms on the
quotient and discreteness of their periods on the regular locus.
These are precisely the assertions of Theorem~\ref{thm:intro:finite}.
\end{proof}

The following example illustrates the approximation theorem
on a singular Stein space, with approximating forms whose
period groups are lattices in $\mathbb C$.

\begin{example}[Singular product approximation with elliptic periods]
\label{ex:singular-product-approximation}
Let $Y$ be any connected normal Stein space (in particular, a singular one), and put
\[
 S=Y\times(\C^*)^2,\qquad
 X=z\partial_z,\qquad
 \omega_\lambda=\frac{dz}{z}+\lambda\frac{dw}{w}.
\]
The forms are closed reflexive differentials on $S$, normalized by $X$, and their period groups are
\[
 \Per(\omega_\lambda)=2\pi i(\Z+\lambda\Z).
\]
Fix an irrational real $\lambda$ and choose $\lambda_j\in\Q+i\Q$ with $\operatorname{Im}\lambda_j\ne0$ and $\lambda_j\to\lambda$. Then $\Per(\omega_\lambda)$ is nondiscrete, whereas each $\Per(\omega_{\lambda_j})$ is a rank-two lattice, and $\omega_{\lambda_j}\to\omega_\lambda$ uniformly on compact sets. For $Y=\{xy=z_0^2\}\subset\C^3$ this gives a genuinely singular normal Stein space exhibiting discrete-period approximation (the coordinate $z_0$ is distinct from the $\C^*$-coordinate $z$). This example is elementary and does not assert a local fiber-bundle structure at singular points.
\end{example}

\begin{remark}[Finite quotients of  complex manifolds]
\label{rem:finite-descent-scope}
Finite quotients of smooth complex manifolds satisfy
the degree-one reflexive Poincar\'e property. In a local linearization
chart, pull a closed reflexive one-form back over the preimage of the
regular locus of the quotient. The complement of this preimage has
codimension at least two, so the coefficients of the pulled-back
holomorphic form extend across it by Hartogs' theorem. The extension
is closed by the identity principle. On a sufficiently small invariant
ball it has a holomorphic primitive. Averaging this primitive over the
finite stabilizer preserves its differential and gives an invariant
primitive, which descends holomorphically to the quotient. This proves
local exactness also at singular points.
\end{remark}

\section{Normal Stein spaces: local period obstructions and flat connections}\label{sec:normal-spaces}
We now examine the local restrictions imposed by normalization, especially near singular points of normal Stein surfaces.

The preceding sections establish approximation results for normal Stein spaces,
finite Stein quotients, and smooth Stein manifolds. We now
examine local existence and obstruction phenomena on normal Stein spaces. Throughout
this section, differentials on a normal space $S$ are understood in the
\emph{reflexive} sense
$\Omega_S^{[1]}=j_*\Omega_{S_{\mathrm{reg}}}^1$, where
$j:S_{\mathrm{reg}}\hookrightarrow S$. Closedness and contraction by
vector fields are interpreted on the regular locus. We make no general
claim that reflexive closed forms compute $H^1(S,\C)$.

\begin{proposition}[Fixed singular points on normal surfaces]\label{prop:normal-fixed-sing}
Let $S$ be a normal complex analytic surface and let a connected
complex Lie group act holomorphically on $S$. Every point of
$\Sing(S)$ is fixed. In particular, a normal surface carrying a
fixed-point-free holomorphic $\C^*$-action is smooth.
This fixed-point observation is classical; compare
\cite[Section~1.3]{OrlikWagreich1971}.
\end{proposition}
\begin{proof}
Normality implies that $\Sing(S)$ has complex codimension at least
two, hence is discrete. Every automorphism preserves $\Sing(S)$.
The orbit of a singular point is the connected image of a connected
Lie group in a discrete space and is therefore a singleton.
\end{proof}

Two elementary consequences of normalization will be useful when analyzing singular orbits.

\begin{lemma}
\label{lem:normalization-obstructions}
Let $S$ be a connected normal complex analytic space
satisfying the degree-one reflexive Poincar\'e property.
Let $X$ generate a holomorphic $\mathbb C^*$-action
on $S$, and suppose that there exists a closed
reflexive holomorphic one-form $\omega$ satisfying
\[
\omega(X)=1
\quad\text{on }S_{\mathrm{reg}}.
\]
Then the action has no fixed points, and every
orbit circle represents a nontorsion class in
$H_1(S,\mathbb Z)$.
\end{lemma}

\begin{proof}
Near a hypothetical fixed point $p$, the degree-one reflexive
Poincar\'e property gives $\omega=df$ for a holomorphic function
$f$. Thus $X(f)=1$ on the regular locus and hence on the whole
neighborhood by normality. At a fixed point, however,
$X(f)(p)=0$, a contradiction.

Let $p\in S_{\mathrm{reg}}$, and consider
the orbit circle
\[
\gamma_p(\theta)=e^{i\theta}\cdot p,
\qquad 0\leq\theta\leq2\pi.
\]
Since $\dot\gamma_p=iX(\gamma_p)$, normalization
gives
\[
\int_{\gamma_p}\omega=2\pi i.
\]
By the degree-one reflexive Poincar\'e property,
the periods of $\omega$ are defined on
$H_1(S,\mathbb Z)$. If $[\gamma_p]$ were torsion,
its period would vanish, a contradiction.
\end{proof}

We first examine how invariant local forms can be obtained by averaging.

\begin{lemma}
\label{lem:local-reflexive-average}
Let $S$ be a normal complex surface with a holomorphic $\C^*$-action and infinitesimal generator $X$. On any $S^1$-invariant open set $V\subset S$, if a reflexive holomorphic one-form $\eta$ satisfies $\eta(X)=1$ on $V\cap S_{\rm reg}$, there exists a reflexive holomorphic one-form $\omega$ on $V$ with
\[
 \omega(X)=1,\qquad d\omega=0\quad\text{on }V\cap S_{\rm reg}.
\]
\end{lemma}
\begin{proof}
Average the pullbacks of $\eta$ over $S^1$. The average defines a holomorphic differential on the regular locus, hence a reflexive differential on the normal surface. It remains normalized, since the action preserves $X$. Circle invariance implies $\mathcal L_X\omega=0$ by differentiation and holomorphicity. Cartan's identity yields $\iota_Xd\omega=0$. Since $\omega(X)=1$, the field is nowhere zero on the regular locus; contraction by a nonzero vector is injective on holomorphic two-forms on a smooth complex surface. Therefore $d\omega=0$.
\end{proof}

The next lemma identifies the role of local Stein geometry in normalization.

\begin{lemma}[Local Stein normalization and Poincar\'e obstruction]
\label{lem:local-stein-normalization}
Let $p$ belong to a normal complex surface $S$ with a holomorphic
$\C^*$-action generated by $X$. If the stalk morphism
$\iota_X:\Omega_{S,p}^{[1]}\to\mathcal O_{S,p}$ is surjective,
then there is a Stein neighborhood of $p$ on which a normalized
reflexive form exists. If $p$ is fixed, one may choose a sufficiently
small $S^1$-invariant Stein neighborhood and obtain a \emph{closed}
normalized reflexive form there. Conversely, existence of a normalized
form in a neighborhood implies surjectivity at $p$. If $p$ is fixed
and the degree-one reflexive Poincar\'e property holds at $p$, the
contraction morphism is not surjective at $p$.
\end{lemma}
\begin{proof}
Surjectivity holds on some neighborhood of $p$, as the cokernel of
the coherent contraction morphism has closed analytic support.
Choose a smaller Stein neighborhood $U$ of $p$. On $U$, the kernel
of contraction is coherent, so Cartan's Theorem B lifts the unit
section to a normalized reflexive form. If $p$ is fixed, average over
$S^1$ after choosing $U$ invariant under this compact group; an
$S^1$-invariant Stein neighborhood basis can be obtained by
averaging a local strictly plurisubharmonic exhaustion and taking
small invariant sublevel sets. Lemma~\ref{lem:local-reflexive-average}
then makes the normalized form closed.
The converse follows by contracting the normalized section.
Finally, a closed normalized form near a fixed point is locally
$df$ under the degree-one reflexive Poincar\'e property, so
$X(f)=1$ on the regular locus and hence everywhere by normality.
At a fixed point $X(f)(p)=0$, a contradiction.
\end{proof}

These local considerations lead to a criterion for global reflexive normalization.

\begin{theorem}[Reflexive normalization criterion]
\label{prop:reflexive-contraction}
Let $S$ be a normal Stein surface with a holomorphic $\C^*$-action,
and let $X$ be its infinitesimal generator. The following conditions are
equivalent:
\begin{enumerate}
\item The coherent contraction morphism
$\iota_X:\Omega_S^{[1]}\longrightarrow\mathcal O_S$ is surjective.
\item There exists $\eta\in H^0(S,\Omega_S^{[1]})$ with
$\eta(X)=1$ on $S_{\rm reg}$.
\item There exists $\omega\in H^0(S,\Omega_S^{[1]})$ whose
restriction to $S_{\rm reg}$ is closed and satisfies $\omega(X)=1$.
\end{enumerate}
\end{theorem}
\begin{proof}
If the contraction morphism is surjective, its kernel is coherent;
Cartan's Theorem B lifts the global unit section
$1\in\mathcal O(S)$ to a global reflexive differential $\eta$.
This proves (1)$\Rightarrow$(2).

Conversely, for a form as in (2), the contraction $\eta(X)$ is a
holomorphic function on $S$ because $S$ is normal and the regular
locus has codimension-two complement. Since it equals $1$ on the
regular locus, it equals $1$ on $S$; multiplication of $\eta$ by
local holomorphic functions therefore provides a sheaf-theoretic
right inverse of contraction. Thus (2)$\Rightarrow$(1).

The local averaging lemma, Lemma~\ref{lem:local-reflexive-average}, applied to $V=S$ gives (3) from (2). Finally, (3)$\Rightarrow$(2) is immediate.
\end{proof}

The period of a closed time form imposes a basic homological constraint on its orbits.

\begin{proposition}\label{prop:orbit-homology}
Let $U^*$ be a smooth complex space carrying a holomorphic vector
field $X$. Suppose $q\in U^*$ has a nonconstant periodic complex-time
orbit $\gamma(t)=\varphi_t(q)$, $t$ running from $0$ to
$\tau\in\C\setminus\{0\}$, and that the class of this loop
vanishes in $H_1(U^*,\C)$. Then no closed holomorphic one-form
$\omega$ on $U^*$ can satisfy $\omega(X)=1$.
\end{proposition}
\begin{proof}
A normalized form would have
$\int_\gamma\omega=\int_0^\tau\omega(X)\,dt=\tau\ne0$.
Integration of a closed complex-valued one-form factors through
$H_1(U^*,\C)$ and therefore vanishes on $[\gamma]$, a
contradiction.
\end{proof}

The preceding homological observation is particularly effective for rational-homology-sphere links.

\begin{corollary}[Rational-homology-sphere links]\label{cor:qh-link}
Let $(S,p)$ be an isolated normal surface singularity with a
holomorphic $\C^*$-action fixing $p$. Choose a sufficiently small
$S^1$-invariant link $L$ about $p$ and a punctured neighborhood
$U^*$ containing it such that $L\hookrightarrow U^*$ is a
homotopy equivalence. Suppose $H_1(L,\Q)=0$ and the circle action
has a nonconstant orbit in $L$. Then there is no closed
holomorphic one-form on $U^*$ normalized by the infinitesimal generator.
No invariance of the bounded neighborhood under the full
$\C^*$-action is assumed.
\end{corollary}
\begin{proof}
The nonconstant circle orbit is a closed complex-time orbit
(with nonzero period), entirely contained in $L\subset U^*$.
Its class is torsion in $H_1(L,\Z)$ and hence zero in
$H_1(U^*,\C)$. Apply Proposition~\ref{prop:orbit-homology}.
\end{proof}

\begin{example}[The quadratic cone: obstruction even without closedness]
\label{ex:quadratic-cone}
Let $S=\{xy=z^2\}\subset\C^3$, with the scalar action
$\lambda\cdot(x,y,z)=(\lambda^2x,\lambda^2y,\lambda^2z)$.
The origin is the unique singularity, and
$S\simeq\C^2/\{\pm I\}$ by $(u,v)\mapsto(u^2,v^2,uv)$.
The regular locus has $H_1(S_{\rm reg},\Z)\simeq\Z/2\Z$;
thus a closed normalized form is excluded by
Proposition~\ref{prop:orbit-homology}. More strongly, no
\emph{reflexive} normalized holomorphic form exists. Indeed, its
pullback to $\C^2\setminus\{0\}$ would be
$A(u,v)\,du+B(u,v)\,dv$ with $uA+vB=1$. Hartogs extension
extends $A,B$ across the origin, making $uA+vB=1$ at the
origin, an impossibility.
\end{example}

The homological obstruction extends far beyond rational homology
sphere links. The Seifert structure and weighted-homogeneous classification
are established results, not new assertions here: see Orlik--Wagreich
\cite{OrlikWagreich1971} in the algebraic setting and Camacho--Movasati--Sc\'ardua
\cite{CamachoMovasatiScardua2009} for good analytic actions on normal Stein surfaces.
Recall that the link of a weighted-homogeneous normal
surface singularity with an isolated fixed point is naturally a
Seifert three-manifold. We use the rational Seifert Euler number
$e_{\rm orb}$, with the sign convention immaterial for its
nonvanishing. If the Seifert data are $(b;(\alpha_i,\beta_i))$,
one convention is $e_{\rm orb}=b+\sum_i\beta_i/\alpha_i$.

\begin{lemma}[Seifert fiber homology]\label{lem:seifert-fiber}
Let $L$ be a closed oriented Seifert three-manifold with oriented
base orbifold, and let $h\in H_1(L,\Z)$ be the class of a regular
circle fiber. If $e_{\rm orb}\ne0$, then $h$ is torsion.
\end{lemma}
\begin{proof}

Let $L$ denote the link of the singularity, endowed
with its natural Seifert fibration. Write
$(g;b;(\alpha_1,\beta_1),\ldots,
(\alpha_r,\beta_r))$ for its Seifert invariants,
where $\alpha_i\geq 2$ is the multiplicity of
the $i$-th exceptional fiber and
$\gcd(\alpha_i,\beta_i)=1$.

Let $h\in H_1(L,\mathbb Z)$ denote the homology
class of an oriented regular fiber, and let
$q_i\in H_1(L,\mathbb Z)$ denote the classes
arising from small loops around the exceptional
points of the base orbifold, lifted using a
section over their complement.

The standard presentation of the fundamental
group of a Seifert fibered three-manifold,
after abelianization, gives the relations
\[
\alpha_i q_i-\beta_i h=0,
\qquad i=1,\ldots,r,
\]
and
\[
q_1+\cdots+q_r+bh=0.
\]
Here the signs are fixed by our choice of
orientations and Seifert invariants; see the
classical presentation of Seifert fibered
manifold groups.

Tensoring with $\mathbb Q$, the first relations
give
\[
q_i=\frac{\beta_i}{\alpha_i}h.
\]
Substitution into the remaining relation yields
\[
\left(
b+\sum_{i=1}^r\frac{\beta_i}{\alpha_i}
\right)h=0
\quad\text{in }H_1(L,\mathbb Q).
\]

Thus, if the rational Seifert Euler number
\[
e_{\mathrm{orb}}
=
b+\sum_{i=1}^r\frac{\beta_i}{\alpha_i}
\]
is nonzero, then $h=0$ in $H_1(L,\mathbb Q)$.
Equivalently, the regular-fiber class is
torsion in $H_1(L,\mathbb Z)$.
 After tensoring with $\Q$, elimination
of the $q_i$ gives
$(b+\sum_i\beta_i/\alpha_i)h=0$ for the
corresponding choice of signs. This coefficient is, up to the
standard sign convention, the rational Euler number.
Its nonvanishing implies $h=0$ in $H_1(L,\Q)$, i.e.
$h$ is torsion integrally.
\end{proof}

The Seifert description of an isolated surface singularity makes the orbit obstruction explicit.

\begin{theorem}[Seifert Euler obstruction]\label{thm:seifert-obstruction}
Let $(S,p)$ be an isolated normal surface singularity with a
holomorphic $\C^*$-action fixing $p$. Assume there is a sufficiently
small $S^1$-invariant Seifert link $L\subset S_{\rm reg}$ and a
punctured representative $U^*$ containing $L$.
Assume that the Seifert circles are the orbits of this $S^1$-action,
and that the induced Seifert fibration has nonzero rational Euler
number $e_{\rm orb}(L)$. Then no closed holomorphic one-form on
$U^*$ is normalized by the infinitesimal generator.
In particular, the conclusion applies to the standard link
fibrations of isolated normal surface singularities with good
(weighted-homogeneous) $\C^*$-actions.
\end{theorem}
\begin{proof}
Choose a point of $L$ on a regular (nonexceptional) Seifert
fiber. Its parametrized circle orbit represents a nonzero integer
multiple of the class $h$ of a regular fiber and, by
Lemma~\ref{lem:seifert-fiber}, is torsion in
$H_1(L,\Z)$. Hence it is torsion in $H_1(U^*,\Z)$.
The inclusion $L\hookrightarrow U^*$ maps torsion homology classes to torsion classes. The same orbit, parametrized by the complex-time holomorphic
flow, has nonzero time period; Proposition~\ref{prop:orbit-homology}
therefore excludes a closed normalized one-form.
For an isolated good $\C^*$-action, an adapted small link
carries the canonical Seifert fibration described by
Orlik--Wagreich \cite[Sections~1--2]{OrlikWagreich1971}; the
analytic setting and equivariant resolution are discussed in
\cite{CamachoMovasatiScardua2009}. For the good-action link, the canonical star-shaped resolution
plumbing has a negative-definite intersection matrix. Eliminating
the negative-definite arm blocks gives a one-by-one Schur complement,
which is nonzero because the full intersection matrix is negative
definite. The standard Seifert plumbing formula identifies this
Schur complement, up to sign, with the rational Euler number
\cite{Neumann1981}. Hence $e_{\rm orb}(L)\ne0$.
For the weighted-action viewpoint see also \cite{Pinkham1977};
compare \cite{Neumann1981,NemethiNicolaescu2004}.
The link is chosen equivariantly, without demanding that a bounded
punctured neighborhood be preserved by the full complex action.
\end{proof}

The same obstruction has a formulation in terms of the contraction morphism.

\begin{corollary}[Non-surjectivity at a good singularity]
\label{cor:good-contraction}
Let $S$ be a normal Stein surface with a good holomorphic
$\C^*$-action having an isolated singular fixed point $p$.
Then the contraction morphism
$\iota_X:\Omega_S^{[1]}\to\mathcal O_S$ is not surjective
at $p$.
\end{corollary}
\begin{proof}
If it were surjective at $p$, coherence and the local definition
of surjectivity would provide, on a neighborhood of $p$, a
reflexive one-form $\eta$ with $\eta(X)=1$ on the regular locus.
After shrinking to a sufficiently small $S^1$-invariant
neighborhood, Lemma~\ref{lem:local-reflexive-average} produces a closed normalized reflexive form on the
punctured neighborhood. This contradicts
Theorem~\ref{thm:seifert-obstruction}.
\end{proof}

The following example illustrates the Seifert Euler
obstruction for a classical isolated surface singularity
whose projectivization is an elliptic curve.

\begin{example}[An elliptic cubic cone]\label{ex:elliptic-cone}
Let $S=\{x^3+y^3+z^3=0\}\subset\C^3$ with the scalar
$\C^*$-action. Its projectivization is a smooth elliptic cubic
$E\subset\mathbb P^2$, and $S^*$ is the complement of the zero
section in $\mathcal O_E(-1)$, of degree $-3$.
The circle fibration has Euler number $-3$, and its homology
sequence yields
$H_1(S^*,\Z)\simeq\Z^2\oplus\Z/3\Z$.
The scalar orbit is the torsion fiber class. Hence a closed
normalized form cannot exist although
$H_1(S^*,\Q)\simeq\Q^2\ne0$.
\end{example}

For comparison, principal bundles over smooth curves admit a direct description in terms of flat connections.

\begin{proposition}
\label{prop:flat-line}
Let $\pi:P\to B$ be a holomorphic principal $\C^*$-bundle over
a connected Riemann surface, with normalized infinitesimal generator $X$.
The following conditions are equivalent:
\begin{enumerate}
\item $P$ has a closed holomorphic one-form $\omega$ with
$\omega(X)=1$;
\item $P$ has a holomorphic principal connection.
\end{enumerate}
If $B$ is compact, these are equivalent to $\deg L=0$ for the
associated line bundle $L$. If $B$ is noncompact, they always hold.
\end{proposition}
\begin{proof}
A closed normalized $\omega$ is invariant by Cartan's formula,
and so is a connection form. Conversely, the curvature of a
holomorphic principal connection is a basic holomorphic two-form,
which vanishes on the one-dimensional base, hence the connection
is closed. For compact $B$ the degree-zero condition is the
classical Atiyah--Weil criterion for holomorphic line bundles.
For noncompact $B$, the Stein and Oka properties of open Riemann
surfaces together with $H^2(B,\Z)=0$ imply that $L$ is
holomorphically trivial; then $dq/q$ is the required form.
\end{proof}

The next example shows that, in complex dimension three,
a singular normal Stein space may admit a fixed-point-free
holomorphic $\C^*$-action together with a closed normalized
reflexive holomorphic one-form.

\begin{example}[A singular Stein threefold with normalized periods]
\label{ex:singular-threefold}
Put $Y=\{xy=z^2\}\subset\C^3$ and $S=Y\times\C^*$.
This is a normal Stein threefold with singular locus
$\{0\}\times\C^*$. The action on the second factor is globally
free, and $X=q\partial_q$ has the global holomorphic reflexive
form $\omega=dq/q$. The form is closed and normalized and has
period group $2\pi i\Z$. Thus the surface rigidity of
Proposition~\ref{prop:normal-fixed-sing} is dimension-sensitive.
\end{example}

\section{Normalized time forms and invariant transverse foliations}\label{sec:hierarchy}
The following elementary criteria relate normalized forms to invariant transverse holomorphic foliations.

\begin{proposition}[Normalization is automatic; closedness is not]\label{prop:CartanB}
Every nowhere-vanishing holomorphic vector field on a Stein manifold admits a global holomorphic one-form $\alpha$ with $\alpha(X)=1$. Even for a complete vector field on a Stein surface, such a form need not be closed, and there need not exist any closed normalized form.
\end{proposition}
\begin{proof}
Contraction by $X$ gives a surjective morphism of locally free coherent sheaves $\Omega_M^1\to\mathcal O_M$, whose kernel is coherent. The induced cohomology sequence and Cartan's Theorem B give a global splitting of $1\in\mathcal O(M)$.
For the second assertion take $M=\C^*\times\C$, with coordinates $(z,w)$, and
\[
X=e^w z\partial_z,\qquad \alpha=e^{-w}\,\frac{dz}{z}.
\]
The flow is $(z,w)\mapsto(ze^{e^w t},w)$ for all $t\in\C$, so $X$ is complete; $\alpha(X)=1$, but $d\alpha=-e^{-w}dw\wedge dz/z\ne0$. For any other normalized one-form $\eta$, restriction to the circle $\gamma_w(\theta)=(e^{i\theta},w)$ necessarily gives $\int_{\gamma_w}\eta=2\pi i e^{-w}$. The circles $\gamma_w$ are homologous while the proposed periods vary with $w$. Therefore no closed normalized $\eta$ exists.
\end{proof}

The following example illustrates how varying orbit periods
may obstruct the existence of a closed normalized holomorphic
one-form.

\begin{example}\label{ex:varyingperiods}
On the Stein surface $M=\C^*\times\C$ take $X=e^wz\partial_z$ and $\alpha=e^{-w}dz/z$. Then $X$ is nonsingular and complete and $\alpha(X)=1$, yet $d\alpha\ne0$. The orbit through $(z,w)$ has fundamental period $2\pi i e^{-w}$, so the periods of a hypothetical closed normalized form around homologous orbit circles would depend on $w$. Therefore $\T_X=\varnothing$. This exhibits why existence of a normalized \emph{holomorphic} form does not suffice for Tischler approximation.
\end{example}

On a complex surface, the invariance condition simplifies considerably.

\begin{proposition}[Invariance criterion on a surface]\label{prop:invclosed}
If $\dim_{\C}M=2$ and $\alpha(X)=1$, then
\[
d\alpha=0\quad\Longleftrightarrow\quad\mathcal L_X\alpha=0\quad\Longleftrightarrow\quad\ker\alpha\text{ is invariant under the local flow of }X.
\]
The last equivalence is for the \emph{normalized} form $\alpha$.
\end{proposition}
\begin{proof}
Cartan's formula gives $\mathcal L_X\alpha=\iota_Xd\alpha$; since $X\ne0$, contraction with $X$ is injective on holomorphic two-forms on a surface. Invariance of the kernel means $\mathcal L_X\alpha=f\alpha$ locally. Evaluation on $X$ gives $f=(\mathcal L_X\alpha)(X)=X(\alpha(X))=0$.
\end{proof}

\section{Surface obstructions and universal covers}
In complex dimension two, invariance of a normalized form has an especially strong consequence for its closedness.

\begin{proposition}\label{torus:prop:surface}
Let $N$ be a connected Stein surface carrying a holomorphic $\C^*$-action without fixed points, and let $X$ be its infinitesimal generator, normalized by
$X(x)=\left.\frac{d}{dt}\right|_{t=0}(e^t\cdot x)$. There exists a holomorphic, closed one-form $\eta$ on $N$ satisfying
\[\eta(X)=1,\qquad (s\cdot)^*\eta=\eta\quad (s\in\C^*).\]
Every orbit circle $\gamma_x(\theta)=e^{i\theta}\cdot x$, $0\leq\theta\leq2\pi$, has period
\[\int_{\gamma_x}\eta=2\pi i.\]
Consequently its homology class has infinite order and $H^1(N,\C)\ne0$.
\end{proposition}
\begin{proof}
Since $X$ has no zero, contraction is a surjective morphism of locally free coherent sheaves
\[\iota_X:\Omega_N^1\longrightarrow\Ocal_N.\]
Its kernel is coherent; Cartan's Theorem B yields $\alpha\in H^0(N,\Omega_N^1)$ with $\alpha(X)=1$. Average over $S^1$:
\[\eta=\frac1{2\pi}\int_0^{2\pi}(e^{i\theta}\cdot)^*\alpha\,d\theta.\]
Then $\eta(X)=1$ and $\eta$ is $S^1$-invariant. Holomorphicity in the group parameter implies invariance under $\C^*$, and hence $\mathcal L_X\eta=0$. Cartan's identity gives
\[\iota_Xd\eta=\mathcal L_X\eta-d(\eta(X))=0.\]
On a complex surface, contraction of a holomorphic two-form by a nonvanishing vector field is injective at each point. Thus $d\eta=0$. Finally, $\dot\gamma_x=iX(\gamma_x)$, so integration gives $2\pi i$. A closed one-form has zero period over a torsion homology class, proving the assertion.
\end{proof}

\begin{remark}[A fundamental group obstruction]
For a connected Stein surface $N$, the existence
of a fixed-point-free holomorphic $\C^*$-action
implies that the center of $\pi_1(N)$ contains
an element of infinite order. Indeed, the
orbit-circle class is central in $\pi_1(N)$
and has infinite order by the existence of
a closed normalized holomorphic one-form.
Consequently, if the center of $\pi_1(N)$
contains no element of infinite order,
every holomorphic $\C^*$-action on $N$
has a fixed point.
\end{remark}

\section{Uniformization of the universal cover}
A complete vector field with a closed normalized time form also constrains the geometry of the universal covering space.

\begin{theorem}[Uniformization from a closed holomorphic time form]
\label{thm:uniform-time}\label{torus:thm:uniform}
Let $N$ be a connected Stein surface and $X$ a nowhere-vanishing
complete holomorphic vector field. Suppose that there exists a closed
holomorphic one-form $\eta$ with $\eta(X)=1$. Then the universal cover
is biholomorphic to $\C\times\Sigma$, where $\Sigma$ is either $\C$
or $\D$. In these coordinates the lifted complete flow is translation,
\[
 \widetilde\varphi_t(z,w)=(z+t,w),
\]
and every deck transformation has the form
\[
 \gamma(z,w)=(z+c_\gamma,\rho_\gamma(w)),\qquad
 c_\gamma\in\C,\quad\rho_\gamma\in\operatorname{Aut}(\Sigma).
\]
If $X$ is the infinitesimal generator of a holomorphic $\C^*$-action, the
deck group moreover contains the nontrivial central translation
$(z,w)\mapsto(z+2\pi i,w)$, so $\pi_1(N)$ contains an infinite
cyclic central subgroup.
\end{theorem}
\begin{proof}
Let $q:\widetilde N\to N$ be the universal covering map. Since $q^*\eta$ is holomorphic and closed, simple connectivity gives a holomorphic primitive $F:\widetilde N\to\C$ with $dF=q^*\eta$. The complete holomorphic flow $\varphi:\C\times N\to N$ generated by $X$ lifts, with identity initial condition, to a holomorphic $\C$-action $\widetilde\varphi$ on $\widetilde N$ because $\C$ is simply connected. Its infinitesimal generator $\widetilde X$ satisfies
\[F(\widetilde\varphi_t(x))=F(x)+t.\]
Thus $\Sigma=F^{-1}(0)$ is a smooth complex curve, and
\[\C\times\Sigma\longrightarrow\widetilde N,\quad (t,y)\longmapsto\widetilde\varphi_t(y)\]
is biholomorphic: its holomorphic inverse sends $x$ to $(F(x),\widetilde\varphi_{-F(x)}(x))$. The product and the simple connectivity of $\widetilde N$ show that $\Sigma$ is connected and simply connected. Hence uniformization gives $\Sigma\simeq\C,\D$, or $\mathbb P^1$. The last case is impossible: composing the inclusion of the compact curve $\mathbb P^1\hookrightarrow\widetilde N$ with $q$ would produce a nonconstant holomorphic map $\mathbb P^1\to N$, whereas holomorphic functions on a Stein manifold separate points and are constant on compact Riemann surfaces.

For any deck transformation $\gamma$, we have $d(F\circ\gamma-F)=0$, hence $F\circ\gamma=F+c_\gamma$. It commutes with the lifted additive action: both lift the same action and agree at $t=0$. The claimed product form follows. If the original flow has global period $2\pi i$, then $\widetilde\varphi_{2\pi i}$ is a deck transformation, and its $F$-increment is $2\pi i$, so it is nontrivial. It commutes with every deck transformation because it is part of the lifted flow.
\end{proof}

The preceding splitting has the following useful exact-time consequence.

\begin{corollary}[Exact-time splitting]\label{cor:exact-time-splitting}
Let $M$ be a connected complex manifold, $X$ a complete holomorphic
vector field, and $T\in\mathcal O(M)$ satisfy $X(T)=1$. Then the
hypersurface $L=T^{-1}(0)$ is connected and
\[
 \C\times L\longrightarrow M,\qquad (t,p)\longmapsto\varphi_t(p)
\]
is a biholomorphism. In these coordinates $X=\partial_t$.
In particular, if $M$ is Stein, $H^1(M,\C)=0$, and
$\mathcal T_X\ne\varnothing$, this splitting holds and every orbit
is biholomorphic to $\C$.
\end{corollary}
\begin{proof}
The equation $X(T)=1$ and completeness yield
$T(\varphi_t(p))=T(p)+t$. Thus the displayed map has holomorphic
inverse $q\mapsto\bigl(T(q),\varphi_{-T(q)}(q)\bigr)$.
Connectedness of $M$ forces connectedness of $L$. In the final
assertion, a normalized closed one-form has vanishing periods when
$H^1(M,\C)=0$, hence has a holomorphic primitive $T$.
\end{proof}

\subsection{Sharpness of the hypotheses}
The examples below indicate which assumptions in the uniformization statement cannot be omitted.

\begin{example}[Steinness cannot be omitted]\label{ex:nonstein}
The scalar action on $\C^2\setminus\{0\}$ is nonsingular, while $H^1(\C^2\setminus\{0\},\C)=0$ because this punctured space retracts to $S^3$. The manifold is not Stein; thus it does not contradict Proposition~\ref{torus:prop:surface}.
\end{example}
\begin{example}[The surface restriction is essential]\label{ex:threefold}
The free right action $A\mapsto A\operatorname{diag}(s,s^{-1})$ on the affine Stein threefold $SL_2(\C)$ has no fixed points, but $H^1(SL_2(\C),\C)=0$ since $SL_2(\C)$ retracts to $SU(2)\simeq S^3$. Therefore the degree-one obstruction is special to the surface or codimension-one torus settings.
\end{example}

\begin{remark}[Steinness matters]
The punctured space $\C^2\setminus\{0\}$ has $H^1(-,\C)=0$ (it retracts onto $S^3$) and has the free scalar $\C^*$-action $t\cdot(x,y)=(tx,ty)$, but it is not Stein. This shows that automatic normalization cannot omit Steinness.
For the uniformization statement itself, $N=\C\times\mathbb P^1$
with $X=\partial_t$ and $\eta=dt$ has all the hypotheses of
Theorem~\ref{torus:thm:uniform} except Steinness, and its universal
cover is $\C\times\mathbb P^1$.
\end{remark}
\begin{remark}[Dimension two matters]
The affine Stein threefold $SL_2(\C)$ has vanishing first complex cohomology. Yet the right action
\[
 A\longmapsto A\,\operatorname{diag}(t,t^{-1}),\qquad t\in\C^*,
\]
is free and holomorphic. Indeed $SL_2(\C)$ retracts onto $SU(2)\simeq S^3$.
\end{remark}

\section{Discrete periods, suspensions, and fiber bundles}
We next make explicit the suspension geometry associated with a closed normalized time form.

\begin{theorem}[Suspension principle]\label{stein:thm:suspension}
Let $M$ be connected and let $X$ be a complete holomorphic vector field admitting $\eta\in\T_X$. Suppose $\Lambda=\Per(\eta)$ is discrete. Then there exist a connected complex manifold $L$ and a homomorphism $\rho:\Lambda\to\Aut(L)$ such that
\[
 M\cong (L\times\C)/\Lambda,
 \qquad \lambda\cdot(p,t)=\big(\rho(\lambda)p,t+\lambda\big).
\]
The flow of $X$ is induced by translation in $t$, and $\eta$ is induced by $dt$. The period integral is a \emph{surjective}, locally holomorphically trivial holomorphic map
\[
 F:M\longrightarrow\C/\Lambda
\]
with \emph{connected} fibers, precisely the leaves of $\ker\eta$.
\end{theorem}
\begin{proof}
Take the connected regular covering $\pi:\widehat M\to M$ corresponding to the kernel of the period homomorphism $\pi_1(M)\to\Lambda$. The pulled-back form is exact, $\pi^*\eta=dT$. The lifted vector field $\widehat X$ is complete and satisfies $\widehat X(T)=1$. Hence $T(\widehat\Phi_t(p))=T(p)+t$. The fiber $L=T^{-1}(0)$ is a complex submanifold and
\[
 L\times\C\to\widehat M,\quad (p,t)\mapsto\widehat\Phi_t(p)
\]
is a biholomorphism, with inverse $q\mapsto(\widehat\Phi_{-T(q)}q,T(q))$. In particular $L$ is connected: the period covering $\widehat M$
is connected by construction and is biholomorphic to $L\times\C$.
Thus connectedness of the fibers follows from the chosen period covering,
not merely from connectedness of the original manifold. Every deck transformation commutes with $\widehat\Phi$ and changes $T$ by $\lambda\in\Lambda$, so it acts as $\big(\rho(\lambda)p,t+\lambda\big)$, with $\rho(\lambda)\in\Aut(L)$. The quotient defines the asserted bundle, and fibers identify with $L$.
\end{proof}
\begin{remark}
The covering/time-coordinate method is classical in nature. The result is retained here solely as the \emph{noncompact structural lemma} needed for the newly derived approximation and obstruction statements, not as a claim of original suspension theory. Without completeness, discrete periods give a submersion but do \emph{not} automatically give a bundle.
\end{remark}

We may now characterize closed leaves by the discreteness of the period group.

\begin{theorem}[Global stability of invariant closed-form foliations]
\label{thm:global-stability}
Let $M$ be a connected complex manifold and $X$ a complete
holomorphic vector field. Let $\omega$ be a closed holomorphic
one-form with $\omega(X)=1$, and put $\mathcal F=\ker\omega$ and
$\Lambda=\operatorname{Per}(\omega)$. The following assertions are
equivalent:
\begin{enumerate}[label=(\roman*),leftmargin=*]
\item $\mathcal F$ has at least one closed leaf;
\item all leaves of $\mathcal F$ are closed in $M$;
\item $\Lambda$ is a discrete subgroup of $\C$;
\item integration of $\omega$ induces a holomorphic fiber bundle
$F:M\to\C/\Lambda$ with connected fibers equal to the leaves of
$\mathcal F$ and $F^*(dz)=\omega$.
\end{enumerate}
Every leaf of $\mathcal F$ has trivial holonomy, independently of
whether these conditions hold. When they do, the base is $\C$,
$\C^*$, or an elliptic curve according as $\Lambda$ has rank
zero, one, or two. No Stein assumption is needed.
\end{theorem}
\begin{proof}
Since $\omega$ is closed, its local holomorphic primitives give
transverse coordinates for $\mathcal F$ and differ on overlaps by
additive constants. Holonomy acts in such coordinates by translations;
as each holonomy germ fixes its basepoint, it is the identity.

The implications (iv)$\Rightarrow$(ii)$\Rightarrow$(i) are clear,
and (iii)$\Rightarrow$(iv) is the suspension principle,
Theorem~\ref{stein:thm:suspension}. To prove the remaining implication,
let $L$ be a closed leaf. The period-cover construction, used in the
proof of Theorem~\ref{stein:thm:suspension}, applies as an abstract
cover even when $\Lambda$ is not discrete as a subset of $\C$.
It gives, for every $p\in L$,
\[
 \{t\in\C:\varphi_t(p)\in L\}=\Lambda.
\]
The left-hand side is closed because $L$ is closed and the orbit map
is continuous; thus $\Lambda$ is closed in $\C$. As $M$ is
second-countable, its fundamental group is countable, so the image
$\Lambda$ of its period homomorphism is countable. A closed
countable subgroup of $\C$ is discrete: otherwise its identity
component or its accumulation at the origin would, by the
classification of closed subgroups of $\R^2$, force an uncountable
subgroup. This proves (i)$\Rightarrow$(iii).
\end{proof}
\begin{remark}[Relation to Reeb stability]
Theorem~\ref{thm:global-stability} is a global noncompact stability statement, rather than
an application of the compact-leaf hypothesis in classical Reeb
stability. Closed holomorphic forms impose trivial holonomy from the
outset, and completeness of the transverse field propagates the
existence of one closed leaf to a global fibration. The holomorphic transverse-action fibration theorem in \cite[Theorem~6]{BehagueScardua2008} contains
the geometric equivalence between a closed leaf and a holomorphic
fibration. The present formulation records the identification of its
transverse return subgroup with $\operatorname{Per}(\omega)$.
\end{remark}

Discrete periods do not replace completeness in the above result:
\begin{example}\label{ex:incomplete-period}
Let $M=\mathbb D\times\C$, $X=\partial_z$ for the disk coordinate $z$, and $\omega=dz$. Then $d\omega=0$, $\omega(X)=1$, and $\operatorname{Per}(\omega)=\{0\}$. The field is not complete on $M$, and integration gives the non-surjective submersion $M\to\C$, $(z,w)\mapsto z$. Hence the bundle conclusion of Theorem~\ref{stein:thm:suspension} requires completeness.
\end{example}
\section{Normalized holomorphic Tischler approximation for \texorpdfstring{$\C^*$}{C*-}actions}
We return to periodic multiplicative actions and derive the corresponding approximation theorem in the smooth setting.

Let $X$ generate a holomorphic $\C^*$-action $\psi$, with the complex-time flow $\Phi_t=\psi_{e^t}$. Let $\gamma_p$ be the orbit circle $\theta\mapsto\psi_{e^{i\theta}}(p)$, $0\le\theta\le2\pi$. Its homology class $\gamma\in H_1(M,\Z)$ does not depend on $p$ when $M$ is connected. (A nontrivial finite isotropy group is permitted as long as $X$ is nowhere zero.)

\begin{theorem}[Affine hyperplane of normalized classes]\label{stein:thm:hyperplane}
Let $M$ be a connected Stein manifold and let $X$ be a nowhere-vanishing complete holomorphic vector field with $\varphi_{2\pi i}=\operatorname{id}_M$ (equivalently, the infinitesimal generator of a fixed-point-free holomorphic $\C^*$-action). Let $\gamma\in H_1(M,\Z)$ be the common homology class of the orbit circles $\theta\mapsto e^{i\theta}\cdot p$. The cohomology classes of closed holomorphic forms $\eta$ with $\eta(X)=1$ are exactly
\[
 \boxed{\;\{\xi\in H^1(M,\C):\ \langle\xi,\gamma\rangle=2\pi i\}\;.}
\]
Every class in the displayed set is realized by a closed holomorphic $1$-form $\eta$ satisfying $\eta(X)=1$. The set is empty if the image of $\gamma$ in $H_1(M,\C)$ is zero; otherwise it is an affine hyperplane.
\end{theorem}
\begin{proof}
Necessity follows from $\int_{\gamma_p}\eta=\int_0^{2\pi}i\,\eta(X)\,d\theta=2\pi i$. Conversely, represent any such class by a global closed holomorphic $1$-form $\alpha$ (holomorphic de Rham theorem on Stein manifolds), and average over the compact circle:
\[
 \bar\alpha=\frac1{2\pi}\int_0^{2\pi}\psi_{e^{i\theta}}^*\alpha\,d\theta.
\]
The average is holomorphic, closed, and preserves the cohomology class because each circle element acts by a map homotopic to the identity. The averaged form is $S^1$-invariant, hence $\Lie_X\bar\alpha=0$ by holomorphic continuation. Cartan's formula gives $d(\bar\alpha(X))=0$, so $\bar\alpha(X)=c$ is constant. Since $2\pi i c=\int_\gamma\bar\alpha=2\pi i$, we have $c=1$.
\end{proof}

\begin{remark}
On Stein surfaces with a nowhere-vanishing infinitesimal generator
of a holomorphic $\C^*$-action, the affine hyperplane in
Theorem~\ref{stein:thm:hyperplane} is nonempty by
Corollary~\ref{stein:cor:surface}.
\end{remark}

The affine description of normalized classes leads to the smooth approximation theorem.

\begin{theorem}[Normalized holomorphic Tischler approximation]\label{stein:thm:Cstar}
Let $M$ be a connected Stein manifold, let $X$ be the nowhere-vanishing generator of a holomorphic $\C^*$-action, and assume that $H_1(M,\Z)$ is finitely generated. Every $\omega\in\T_X$ is a compact-open limit of closed holomorphic $\omega_j\in\T_X$ with $\Per(\omega_j)$ discrete. If $b_1(M)\ge2$, the approximations can have rank-two lattice period groups; hence they define bundles over elliptic curves. All resulting foliations are invariant under the \emph{same} complete holomorphic flow.
\end{theorem}
\begin{proof}
Fix $\omega\in\T_X$ and put $H_1(M,\Z)/\mathrm{tors}\simeq\Z^r$. As $\int_\gamma\omega=2\pi i$, the orbit class $\gamma$ has infinite order, so $r\ge1$. By Smith normal form, choose an integral basis $e_1,\ldots,e_r$ with $[\gamma]=d e_1$ modulo torsion, where $d\ge1$. Under the period description of $H^1(M,\C)$, the normalization hyperplane of Theorem~\ref{stein:thm:hyperplane} is exactly
\[
 \xi(e_1)=2\pi i/d,\qquad \xi(e_j)\in\C\quad (2\le j\le r).
\]
Approximate each of the remaining $r-1$ coordinates of $[\omega]$ by elements of $2\pi i(\Q+i\Q)$, keeping $\xi(e_1)=2\pi i/d$ fixed. Call the resulting characters $\xi_n$. The period values of any given $\xi_n$ lie in $\frac{2\pi i}{N_n}\Z[i]$ for a suitable integer $N_n\ge1$; hence their image is a discrete subgroup of $\C$. If $r\ge2$, we can arrange, however close the approximation, that $\xi_n(e_2)$ has nonzero real part. Together with the fixed nonzero purely imaginary value of $\xi_n(e_1)$, the image then has real span of dimension two and is a rank-two lattice. If $r=1$, normalization forces the single free period value; it is already discrete.

Choose closed global holomorphic forms $\theta_1,\ldots,\theta_r$ representing a cohomology basis, as permitted by holomorphic de Rham theory on Stein manifolds. Because $\xi_n-[\omega]\to0$, suitable linear combinations of the $\theta_j$ provide closed forms $\alpha_n\to\omega$ uniformly on compact subsets with $[\alpha_n]=\xi_n$. Now average $\alpha_n$ over $S^1$:
\[
 \omega_n=\frac1{2\pi}\int_0^{2\pi}\psi_{e^{i\vartheta}}^*\alpha_n\,d\vartheta.
\]
The proof of Theorem~\ref{stein:thm:hyperplane} shows that $\omega_n(X)=1$ exactly and that $[\omega_n]=\xi_n$, so its periods are discrete (and a rank-two lattice if $r\ge2$). Since $\omega$ is invariant and the saturation of any compact set under $S^1$ is compact, the averaging preserves compact-open convergence to $\omega$. Theorem~\ref{stein:thm:suspension} gives the asserted holomorphic fibrations.
\end{proof}

The first Betti number sharpens the conclusion in three natural cases.

\begin{corollary}[The first-Betti-number trichotomy]\label{cor:tischler-betti}
Let $M$ be a connected Stein manifold endowed with a holomorphic
$\C^*$-action whose infinitesimal generator $X$ is nowhere zero, and
suppose that $H_1(M,\Z)$ is finitely generated. Then:
\begin{enumerate}[label=\textup{(\roman*)},leftmargin=*]
\item If $b_1(M)=0$, then $\T_X=\varnothing$.
\item If $b_1(M)=1$ and $\T_X\ne\varnothing$, every $\omega\in\T_X$
 already has a nonzero cyclic discrete period group. Integration of $\omega$
 gives a surjective, locally holomorphically trivial holomorphic map
 $F_\omega:M\to\C^*$ with connected fibers equal to the leaves of
 $\ker\omega$; no approximation is required.
\item If $b_1(M)\ge2$ and $\T_X\ne\varnothing$, every $\omega\in\T_X$
 is a compact-open limit of normalized closed holomorphic forms whose
 periods are rank-two lattices; the associated leaf foliations are
 holomorphic fiber bundles over elliptic curves.
\end{enumerate}
In particular, the first case cannot occur for a nowhere-vanishing
$\C^*$-generator on a Stein \emph{surface}, because such generators have
$\T_X\ne\varnothing$ by Corollary~\ref{stein:cor:surface}.
\end{corollary}
\begin{proof}
Fix $p\in M$ and consider the orbit circle
$\gamma(\theta)=e^{i\theta}\cdot p$. If $\omega\in\T_X$ then
$\int_\gamma\omega=2\pi i$, whence $[\gamma]$ is nontorsion and
$b_1(M)\ge1$. This proves (i).

Suppose $b_1(M)=1$. Because the first integral homology is finitely
generated, its free part is $\Z e$. Write $[\gamma]=d e$ modulo
torsion, with a nonzero integer $d$ (choose the sign of $e$ so that
$d>0$). Normalization forces
\[
 \int_e\omega=\frac{2\pi i}{d},\qquad
 \Per(\omega)=\frac{2\pi i}{d}\Z.
\]
The complete-flow suspension Theorem~\ref{stein:thm:suspension}
therefore gives a holomorphic bundle over
$\C/((2\pi i/d)\Z)\simeq\C^*$ with connected fibers.
This proves (ii). Assertion (iii) is Theorem~\ref{stein:thm:Cstar}.
The final assertion follows from Corollary~\ref{stein:cor:surface}.
\end{proof}

\begin{proof}[Proof of Corollary~\ref{thm:intro:a}]
The three cases and their period assertions follow from
Corollary~\ref{cor:tischler-betti}. When $b_1(M)=1$,
the associated discrete cyclic periods yield a bundle over
$\C^*$ by Theorem~\ref{stein:thm:suspension}; when
$b_1(M)\ge2$, Theorem~\ref{stein:thm:Cstar} provides rank-two
lattice approximants whose foliations are bundles over elliptic
curves, again by the suspension theorem. Finally,
Theorem~\ref{stein:thm:hyperplane} identifies the cohomology
classes of normalized forms with the stated affine hyperplane.
\end{proof}

In complex dimension two, a fixed-point-free multiplicative action also supplies the required closed time form.

\begin{corollary}[Automatic normalization on Stein surfaces]\label{stein:cor:surface}
If $M$ is a Stein surface and $X$ is a nowhere-vanishing generator of a holomorphic $\C^*$-action, there exists $\omega\in\T_X$. If $H_1(M,\Z)$ is finitely generated, then $M$ admits a holomorphic fiber bundle over $\C^*$ or an elliptic curve whose fibers are connected, Stein, and transverse to $X$; when $b_1(M)\ge2$, the base can be elliptic.
\end{corollary}
\begin{proof}
Contraction with $X$ defines a surjection of locally free sheaves $\Omega_M^1\to\OO_M$. Cartan B lifts $1$ to a global holomorphic form $\alpha$ with $\alpha(X)=1$. Average $\alpha$ over $S^1$ to obtain invariant $\bar\alpha$ with $\bar\alpha(X)=1$. Then $\ii_Xd\bar\alpha=\Lie_X\bar\alpha-d(1)=0$. On a complex surface, contraction of a holomorphic $2$-form by a nowhere-zero vector field is injective, so $d\bar\alpha=0$. Apply Theorems~\ref{stein:thm:Cstar} and~\ref{stein:thm:suspension}. A closed complex submanifold of Stein $M$ is Stein.
\end{proof}

\begin{proof}[Proof of Theorem~\ref{thm:intro:b}]
By Proposition~\ref{torus:prop:surface}, the fixed-point-free
holomorphic $\C^*$-action on the Stein surface admits a closed
holomorphic one-form $\omega$ normalized by $X$, and its orbit
circle has nonzero period. Thus $H^1(M,\C)\ne0$ and
$b_1(M)\ge1$. Theorem~\ref{thm:uniform-time} proves the claimed
universal-cover splitting and translation form of the lifted flow.
If $H_1(M,\Z)$ is finitely generated, apply
Corollary~\ref{cor:tischler-betti}: for $b_1(M)=1$ the normalized
form already has cyclic discrete periods, while for $b_1(M)\ge2$
we choose one of the rank-two lattice approximants.
Corollary~\ref{cor:reflexive-fibration} (or, equivalently,
Theorem~\ref{stein:thm:suspension} in the smooth case) gives
holomorphic locally trivial bundles with connected Stein fibers
and the respective bases $\C^*$ and an elliptic curve.
As $M$ is a surface, each fiber has complex dimension one.
\end{proof}

\begin{remark}
The averaging proof is independent of the much deeper zero-free form approximation theorem of Majcen \cite{Majcen2007}. It is the \emph{fixed normalization} $\omega_j(X)=1$ that forces nowhere-vanishing, not an independent noncritical-form construction.
\end{remark}

\Needspace{13\baselineskip}
\begin{example}[Betti-number $b_1=0$: a free action with no closed time form]\label{ex:betti-zero}
Let $M=SL_2(\C)$, the affine Stein threefold
$\{A\in\operatorname{Mat}_{2\times2}(\C):\det A=1\}$, and let
\[
 s\cdot A=A\begin{pmatrix}s&0\\0&s^{-1}\end{pmatrix},
 \qquad s\in\C^*.
\]
This action is free, and its nowhere-zero generator is
$X(A)=AH$, where $H=\operatorname{diag}(1,-1)$.
The manifold retracts onto $SU(2)\simeq S^3$, hence $b_1(M)=0$.
Consequently $\T_X=\varnothing$ by
Corollary~\ref{cor:tischler-betti}: in fact, if a closed normalized
form $\eta$ existed, its integral along the orbit circle would be
$2\pi i$ even though $H^1(M,\C)=0$.

The obstruction is specifically to \emph{closed} normalization:
letting $\vartheta=A^{-1}dA$ be the holomorphic Maurer--Cartan form,
\[
 \alpha=\tfrac12\operatorname{tr}(H\vartheta)
\]
is a global holomorphic form with $\alpha(X)=1$. It is not closed.
Indeed, if $E=\left(\begin{smallmatrix}0&1\\0&0\end{smallmatrix}\right)$
and $F=\left(\begin{smallmatrix}0&0\\1&0\end{smallmatrix}\right)$,
then $[E,F]=H$ and the Maurer--Cartan formula at the identity gives
$d\alpha(E,F)=-\alpha(H)=-1$.
Thus the $b_1=0$ case is genuinely realized by a nonsingular
$\C^*$-action on a Stein manifold, but not by a Stein surface.
\end{example}

\begin{example}[First Betti number one: an existing fibration]\label{ex:betti-one}
Take $M=\C^*\times\C$ with coordinates $(z,w)$ and the free action
$s\cdot(z,w)=(sz,w)$. Then $b_1(M)=1$ and
\[
 X=z\partial_z,\qquad \omega=\frac{dz}{z},\qquad
 \Per(\omega)=2\pi i\Z.
\]
The form is holomorphic, closed, and normalized by $X$.
The associated holomorphic fibration is already the projection
\[
 F:M\longrightarrow\C^*,\qquad F(z,w)=z,
\]
whose connected fibers are the copies $\{z\}\times\C$, precisely
 the leaves of $\ker\omega$. No perturbation is necessary.
\end{example}

\begin{example}[First Betti number two: elliptic approximation of a nonclosed foliation]\label{ex:tischler-explicit}
On $M=(\C^*)^2$, where $b_1(M)=2$, take the free action
$s\cdot(z,w)=(sz,w)$, whose generator is $X=z\partial_z$, and put
\[
 \omega=\frac{dz}{z}+\sqrt2\frac{dw}{w},\qquad
 \omega_n=\frac{dz}{z}+\left(\sqrt2+\frac{i}{n}\right)\frac{dw}{w}\quad(n\ge1).
\]
All these forms are closed and normalized by $X$, and $\omega_n\to\omega$ compact-open. The initial period group $2\pi i(\Z+\sqrt2\Z)$ is dense in $i\R$, whereas the periods of $\omega_n$ form the rank-two lattice
\[
 2\pi i\Z+2\pi i\left(\sqrt2+\frac{i}{n}\right)\Z.
\]
Thus every $\omega_n$ defines a holomorphic bundle over an elliptic curve with fibers transverse to the same complete field $X$, while the original invariant foliation has nonclosed leaves.
\end{example}
\section{The closures of invariant holomorphic leaves}
Period groups also govern the topology of leaf closures for invariant holomorphic foliations.

\begin{theorem}[Leaf-closure classification]\label{stein:thm:closures}
Let $M$ be connected, let $X$ be complete, and let $\omega\in\T_X$. Write $\Lambda=\Per(\omega)$ and $H=\overline\Lambda\subset\C$. There is a continuous surjection
\[
 Q:M\to\C/H
\]
such that, for every leaf $L$ of $\ker\omega$,
\[
 \boxed{\overline L=Q^{-1}(Q(L)).}
\]
Exactly one of these alternatives holds:
\begin{enumerate}[label=(\roman*),leftmargin=*]
\item $H$ is discrete: every leaf is closed, and there is a holomorphic fiber bundle $M\to\C/H$ whose connected fibers are leaves.
\item $H$ is proper and nondiscrete: every leaf is dense in a proper, closed, real-analytic Levi-flat hypersurface. The hypersurfaces are fibers of a real-analytic submersion onto $\R$ or $S^1$.
\item $H=\C$: every leaf is dense in $M$.
\end{enumerate}
\end{theorem}
\begin{proof}
On the period covering, use the splitting $\widehat M\simeq L_0\times\C$ from the proof of Theorem~\ref{stein:thm:suspension}, which does not require the period image to be discrete as a subset of $\C$: the deck group is an abstract quotient of $\pi_1(M)$ acting properly on $\widehat M$, and acts by translation by $\lambda\in\Lambda$ on the time coordinate. Thus $Q([p,t])=t\pmod H$ is well defined, continuous, and surjective.

For a fixed leaf $L$ of the foliation, $\lambda\in\Lambda$ precisely when $\Phi_\lambda(L)=L$ (use the product cover). If $h\in H$, take $\lambda_n\in\Lambda$ with $\lambda_n\to h$; then $\Phi_{\lambda_n}(p)\in L$ tends to $\Phi_h(p)$ for each $p\in L$. The flow saturates $M$ by translates of $L$, and therefore the entire fiber $Q^{-1}(Q(L))$ lies in $\overline L$. Conversely $Q$ is constant on $L$ and has closed point fibers since $H$ is closed, giving the reverse inclusion.

A proper nondiscrete closed subgroup of the real plane $\C$ is either a real line or a real line plus a transverse infinite cyclic subgroup. Its quotient is respectively $\R$ or $S^1$. Locally $Q$ is the real part of a complex-linear multiple of a holomorphic time coordinate, modulo a real period, hence is a real-analytic submersion with Levi-flat level hypersurfaces.
\end{proof}

\begin{proof}[Proof of Theorem~\ref{thm:intro:f}]
The transverse splitting $TM=T\mathcal F\oplus\C X$ defines
a unique holomorphic one-form $\omega$ with
$\ker\omega=T\mathcal F$ and $\omega(X)=1$.
Invariance gives $\mathcal L_X\omega=0$, hence
$\iota_Xd\omega=0$ by Cartan's formula. Integrability gives
$d\omega(Y,Z)=-\omega([Y,Z])=0$ for sections $Y,Z$ of
$T\mathcal F$. These identities imply $d\omega=0$.
Thus the hypotheses of Theorem~\ref{stein:thm:closures} hold.
Apply the period-cover construction and the classification
of closed subgroups of $\C$ in that theorem. Discrete closure of
the period group gives closed leaves; a proper nondiscrete
closure gives real-analytic Levi-flat leaf closures; and
$H=\C$ gives dense leaves. None of these steps requires
Steinness.
\end{proof}

The same period-cover method extends the leaf-closure description to normal complex spaces.

\begin{corollary}[Leaf closures on normal complex spaces]
\label{cor:normal-leaf-closures}
Let $S$ be a connected normal complex analytic space satisfying the
 degree-one reflexive Poincar\'e property. Let $X$ be a complete
 holomorphic vector field on $S$, with holomorphic flow
 $\Phi:\C\times S\to S$, and let
 $\omega\in H^0(S,\mathcal Z_S^1)$ be a closed reflexive holomorphic
 one-form satisfying $\omega(X)=1$ on $S_{\rm reg}$.
 Write $\mathcal F_{\rm reg}=\ker(\omega|_{S_{\rm reg}})$,
 let $\Lambda=\Per_S(\omega)$ be the image of the period character
 defined by local primitives on $S$, and set
 $H=\overline\Lambda\subset\C$.
 Then there is a continuous surjection
 \[
 Q:S\longrightarrow\C/H
 \]
 such that, for every leaf $L$ of $\mathcal F_{\rm reg}$,
 \[
 \boxed{\overline L^{\,S}=Q^{-1}(Q(L)).}
 \]
 More precisely:
 \begin{enumerate}[label=(\roman*),leftmargin=*]
 \item If $H$ is discrete, $Q$ is a holomorphic locally trivial
 fiber bundle with connected normal analytic fibers. Each fiber is
 the closure in $S$ of one regular leaf, and its regular locus is
 that leaf.
 \item If $H$ is proper and nondiscrete, $\C/H$ is isomorphic as a
 real-analytic manifold to $\R$ or $S^1$. Locally in $S$, the map
 $Q$ is induced by a real-linear functional of a holomorphic time
 coordinate. Its fibers are closed real-analytic subsets of real
 codimension one, whose intersections with $S_{\rm reg}$ are
 Levi-flat real-analytic hypersurfaces; the fibers may contain
 singular points of $S$.
 \item If $H=\C$, every leaf of $\mathcal F_{\rm reg}$ is dense in
 $S$.
 \end{enumerate}
 Neither Steinness nor finite generation of $H_1(S,\Z)$ is assumed.
\end{corollary}
\begin{proof}
 The local primitive property gives a period homomorphism
 $\pi_1(S)\to\C$ with image $\Lambda$. On the connected covering
 $\widehat S\to S$ corresponding to its kernel, the pullback of
 $\omega$ has a global holomorphic primitive
 $T:\widehat S\to\C$. Normalization and completeness lift to
 $T(\widehat\Phi_t(q))=T(q)+t$. As in the suspension argument,
 \[
 \widehat S\simeq N\times\C,\qquad N=T^{-1}(0),
 \]
 by $(n,t)\mapsto\widehat\Phi_t(n)$. The space $N$ is connected
 and normal: the covering is connected and normal, and the product
 isomorphism implies both assertions for $N$. In particular,
 $N_{\rm reg}$ is connected and dense in $N$.

 A deck transformation indexed by $\lambda\in\Lambda$ acts in
 these coordinates as
 \[
 (n,t)\longmapsto(\rho_\lambda(n),t+\lambda),
 \]
 where $\rho_\lambda$ is a biholomorphism of $N$. Thus
 $Q([n,t])=t+H$ is well defined, continuous and surjective.
 The regular leaves are the images of slices
 $N_{\rm reg}\times\{t\}$: they are connected because
 $N_{\rm reg}$ is connected, and the product decomposition
 identifies the regular locus of $\widehat S$ with
 $N_{\rm reg}\times\C$.

 Fix such a leaf $L$ and a point $[n,t+h]\in Q^{-1}(Q(L))$,
 where $h\in H$ and $t$ is the slice parameter of $L$.
 Choose $\lambda_j\in\Lambda$ with $\lambda_j\to h$.
 Since the slice $N_{\rm reg}\times\{t+\lambda_j\}$ projects
 onto the same leaf $L$, and $N_{\rm reg}$ is dense in $N$,
 one can choose $n_j\in N_{\rm reg}$ tending to $n$.
 The projected points $[n_j,t+\lambda_j]\in L$ converge to
 $[n,t+h]$. Hence $Q^{-1}(Q(L))\subseteq\overline L^{\,S}$.
 The reverse inclusion follows because $Q$ is continuous,
 $H$ is closed and $Q$ is constant on $L$.

 If $H$ is discrete, the local primitives of $\omega$ glue modulo
 $H=\Lambda$ to a holomorphic map $Q:S\to\C/\Lambda$.
 The product decomposition and local sections of
 $\C\to\C/\Lambda$ make it a locally trivial holomorphic bundle.
 The fibers identify with $N$, so are connected and normal, and
 their regular loci are precisely the regular leaves.

 If $H$ is proper and nondiscrete, classification of closed
 subgroups of $\C\simeq\R^2$ gives $\C/H\simeq\R$ or $S^1$.
 In a local holomorphic time coordinate $f$ provided by a primitive
 of $\omega$, the quotient map is locally represented by a
 nonzero real-linear functional of $f$ (possibly modulo a real
 period). Thus its level sets are locally products of a normal
 analytic space with a real line in the time disc. On the regular
 locus they are real-analytic Levi-flat hypersurfaces, while in
 $S$ they are real-analytic subsets of real codimension one.
 The case $H=\C$ follows immediately from the closure identity.
\end{proof}

As a consequence, the first Betti number restricts the possible leaf-closure behavior.

\begin{corollary}\label{stein:cor:betti}
Under the hypotheses of Theorem~\ref{stein:thm:closures}, if $H_1(M,\Z)$ is finitely generated, then $b_1(M)\le1$ forces closed leaves; $b_1(M)=2$ permits closed or intermediate Levi-flat closures but not full density. Full density requires $b_1(M)\ge3$. On a Stein manifold of complex dimension at least two, closed positive-dimensional holomorphic leaves are noncompact.
\end{corollary}
\begin{proof}
$\Lambda$ is generated by at most $b_1(M)$ complex numbers. An abelian subgroup of $\C$ generated by at most two elements is either discrete or lies in a real line, where its closure may be nondiscrete. Density in $\C$ needs at least three generators. A compact positive-dimensional complex submanifold cannot sit inside a Stein manifold.
\end{proof}
\begin{example}[Intermediate Levi-flat closures]\label{ex:intermediate}
On $M=(\C^*)^2$, set
\[
 X=z\partial_z,\qquad \omega=\frac{dz}{z}+\sqrt2\frac{dw}{w}.
\]
Then $\Lambda=2\pi i(\Z+\sqrt2\Z)$ and $H=i\R$. Each leaf has closure
\[
 \boxed{\;\{(z,w):\log|z|+\sqrt2\log|w|=c\}\;}
\]
for an appropriate real $c$; these are proper Levi-flat hypersurfaces.
\end{example}
\begin{example}[Every transverse leaf is dense]\label{ex:dense}
On $M=\C^*\times(\C\setminus\{0,1\})$, let
\[
 X=z\partial_z,\quad
 \omega=\frac{dz}{z}+i\frac{dw}{w}+(\sqrt2+i\sqrt3)\frac{dw}{w-1}.
\]
Its period image is $2\pi i[\Z+i\Z+(\sqrt2+i\sqrt3)\Z]$, dense in $\C$ by Kronecker's theorem; hence all leaves are dense in $M$.
\end{example}

\section{Additive actions and generic counterexamples}
Additive actions illustrate why the existence of a normalized time form need not lead to discrete-period approximation.

\begin{theorem}[Stein orbit quotient: positive approximation]\label{stein:thm:Steinquotient}
Suppose a complete holomorphic $X$ generates a free proper $\C$-action on Stein $M$, and the quotient Riemann surface $B=M/\C$ is Stein. If $H_1(M,\Z)$ is finitely generated, every $\omega\in\T_X$ admits normalized discrete-period approximation.
\end{theorem}
\begin{proof}
The projection $\pi:M\to B$ is a principal $\C$-bundle. Its class in $H^1(B,\OO_B)$ vanishes because $B$ is Stein, so $M\cong B\times\C$ and $X=\partial_t$. Every normalized closed form is $dt+\pi^*\alpha$, with $d\alpha=0$. Approximate $[\alpha]\in H^1(B,\C)$ by discrete-period classes in the finite-dimensional cohomology space and realize small corrections by global closed holomorphic forms on Stein $B$. The forms $dt+\pi^*\alpha_j$ are normalized and converge compact-open.
\end{proof}

We identify the relevant period constraints on a suspension through the cohomology of its base.

\begin{proposition}[Identification with quotient Dolbeault cohomology]\label{stein:prop:Dolbeault}
Let $\pi:M\to B$ be a holomorphic principal $\C$-bundle over a connected
Riemann surface, and assume that the total space $M$ is Stein. Let $X$ be
its infinitesimal generator and let $\OO_X$ denote the sheaf of local
holomorphic first integrals of $X$. Then there are canonical isomorphisms
\[
 \OO_X\simeq\pi^{-1}\OO_B,\qquad
 \mathcal Q_X:=\OO(M)/X\OO(M)
 \simeq H^1(M,\OO_X)\simeq H^1(B,\OO_B).
\]
If $B$ is compact of genus $g$, Dolbeault theory identifies the last group
with $H^{0,1}(B)$, of complex dimension $g$.
\end{proposition}
\begin{proof}
We first justify the identification of sheaves and the passage from total
space to base, rather than relying only on the contractibility of the fiber.
On every holomorphic bundle chart $\pi^{-1}(U)\simeq U\times\C$,
write the fiber coordinate as $t$. Then $X=\partial_t$, and a holomorphic
function is killed by $X$ precisely when it is locally the pullback of
a holomorphic function of the base coordinate. This gives the asserted
sheaf isomorphism $\OO_X\simeq\pi^{-1}\OO_B$. Moreover the local flow-box
primitive gives an exact sequence of sheaves of complex vector spaces
\begin{equation}\label{eq:OXlocalresolution}
 0\longrightarrow\OO_X\longrightarrow\OO_M
   \xrightarrow{\,X\,}\OO_M\longrightarrow0.
\end{equation}
Since $M$ is Stein, the long exact cohomology sequence and
Cartan's Theorem~B applied to $\OO_M$ imply
\[
 H^1(M,\OO_X)\simeq\OO(M)/X\OO(M)=\mathcal Q_X.
\]

For the base-cohomology identification, take a sufficiently small Stein
coordinate disc $U\subset B$ over which the bundle is holomorphically
trivial. The total space $\pi^{-1}(U)\simeq U\times\C$ is Stein.
For an arbitrary $h\in\OO(U\times\C)$, the holomorphic function
\[
 H(z,t)=\int_0^t h(z,s)\,ds
\]
is globally well defined, and $\partial_t H=h$.
Applying cohomology to~\eqref{eq:OXlocalresolution} on $U\times\C$,
using this surjectivity and Cartan's Theorem~B, gives
\[
 H^q\bigl(\pi^{-1}(U),\OO_X\bigr)=0\qquad(q\ge1).
\]
Also, the $X$-invariant holomorphic functions on the entire fibers over
$U$ are exactly pullbacks of $\OO_B(U)$, so
$\pi_*\OO_X\simeq\OO_B$.
The displayed local vanishing proves $R^q\pi_*\OO_X=0$ for $q\ge1$.
The Leray spectral sequence therefore degenerates to yield
\[
 H^1(M,\OO_X)\simeq H^1(B,\OO_B).
\]
If $B$ is compact, the Dolbeault isomorphism
$H^1(B,\OO_B)\simeq H^{0,1}(B)$ completes the proof.
\end{proof}

The next result shows that this constraint can exclude discrete periods generically.

\begin{theorem}[Generic failure of normalized discrete periods]\label{stein:thm:counterexample}
Let $B$ be \emph{any} compact Riemann surface of genus $g\ge3$, with
universal covering $\D\to B$ and deck group $\Gamma$.
There exists a residual subset
\[
 \mathscr G_B\subset\operatorname{Hom}(\Gamma,\R)
    \simeq H^1(B,\R),
\]
whose complement has Lebesgue measure zero, such that for every
$\rho\in\mathscr G_B$ the complex surface
\[
 M_\rho=(\D\times\C)/\Gamma,
 \qquad \gamma\cdot(z,t)=(\gamma z,t+\rho(\gamma)),
\]
is connected and Stein, and its translation vector field $X$ is complete
and generates a free, proper holomorphic $\C$-action. Moreover,
\[
 H_1(M_\rho,\Z)\simeq\Z^{2g},\qquad \T_X\ne\varnothing,
\]
but \emph{no} $\eta\in\T_X$ has a discrete period group:
\[
 \boxed{\;\{\eta\in\T_X:\operatorname{Per}(\eta)
       \text{ is discrete}\}=\varnothing.\;}
\]
In particular, normalized discrete-period approximation fails at every
$\eta\in\T_X$. For $g=3$ this gives the original genus-three example.
\end{theorem}
\begin{proof}
Fix the compact surface $B$. Since $g\ge3$, its universal covering is
$\D$ and $\Gamma\simeq\pi_1(B)$. A real character
$\rho:\Gamma\to\R$ defines the indicated holomorphic principal
$\C$-bundle $\pi:M_\rho\to B$. Translation in $t$ commutes with the
$\Gamma$-action, so it descends to a complete, free holomorphic
$\C$-action; local triviality over the Hausdorff base makes the action
proper. As the fibers are contractible, the bundle projection is a
homotopy equivalence, and
$H_1(M_\rho,\Z)\simeq H_1(B,\Z)\simeq\Z^{2g}$.

The coordinate differential $dt$ is $\Gamma$-invariant and descends to
a closed holomorphic form $\omega_\rho$ with $\omega_\rho(X)=1$.
Under the canonical identification
$H^1(M_\rho,\C)\simeq H^1(B,\C)$, its cohomology class corresponds
to the class of the additive character $\rho$ (up to the chosen deck
convention). The character $\rho$ is nontrivial on the generic set under
consideration; the suspension is then exactly the universal
globalization of the compact-hyperbolic-base model of
Iannuzzi--Trapani. Thus $M_\rho$ is Stein by
\cite[Proposition~4.1(i)]{IannuzziTrapani2012}; see also
\cite[Remark~4.3]{IannuzziTrapani2012}.

We now determine \emph{all} normalized closed holomorphic forms.
If $\eta\in\T_X$, then $\beta=\eta-\omega_\rho$ satisfies
$d\beta=0$ and $\beta(X)=0$. Cartan's formula gives
$\mathcal L_X\beta=0$. Thus $\beta$ is basic for the principal bundle,
and descends to a holomorphic one-form $\alpha$ on $B$. Conversely
any such form may be pulled back. Hence
\begin{equation}\label{eq:all-normalized-on-Mrho}
 \boxed{\;\T_X=\omega_\rho+\pi^*H^0(B,\Omega_B^1).\;}
\end{equation}

Consider the Hodge decomposition
\[
 H^1(B,\C)=H^{1,0}(B)\oplus H^{0,1}(B),
 \qquad\dim_\C H^{0,1}(B)=g,
\]
and let $q:H^1(B,\C)\to H^{0,1}(B)$ be projection.
Every $\eta\in\T_X$ has the same antiholomorphic component:
\begin{equation}\label{eq:fixed-hodge}
 q([\eta])=q([\rho]),
\end{equation}
where the classes on $M_\rho$ have been identified with classes on $B$.

If $\operatorname{Per}(\eta)\subset\C$ is discrete, it is a free
abelian group of rank at most two. Choose generators $a,b\in\C$
(with one or both zero if the rank is smaller). The integral coordinate
characters of the period homomorphism provide
$u,v\in H^1(B,\Z)$ with
\[
 [\eta]=a\,u+b\,v,
 \qquad q([\rho])=a\,q(u)+b\,q(v).
\]
Thus $q([\rho])$ belongs to the complex vector space spanned by
$q(u)$ and $q(v)$.

Define the exceptional set
\[
 \mathscr E_B=
 \bigcup_{u,v\in H^1(B,\Z)}
 \operatorname{span}_\C\{q(u),q(v)\}
 \ \subset H^{0,1}(B).
\]
Because $g\ge3$, each constituent has complex dimension at most two
and is therefore a proper closed linear subspace of $H^{0,1}(B)$.
There are countably many pairs $(u,v)$, hence $\mathscr E_B$ is
meagre and has Lebesgue measure zero.
The restriction
\[
 q:H^1(B,\R)\longrightarrow H^{0,1}(B)
\]
is an isomorphism of \emph{real} vector spaces: an element of
$H^{1,0}(B)$ which is also real must vanish, and both real dimensions
are $2g$. Consequently
\[
 \mathscr G_B
 =H^1(B,\R)\setminus q^{-1}(\mathscr E_B)
\]
is residual and of full measure. Since $0\in\mathscr E_B$, every
$\rho\in\mathscr G_B$ is nonzero, so Steinness applies. Equation
\eqref{eq:fixed-hodge} contradicts discreteness of the periods for
any $\eta\in\T_X$. This completes the proof.
\end{proof}
\begin{example}[Genus-three Stein principal bundle]\label{ex:genus3}
Let $B=\D/\Gamma$ be compact of genus three and let $\rho:\Gamma\to\R$ belong to the residual full-measure set $\mathscr G_B$ constructed in Theorem~\ref{stein:thm:counterexample}. Then
$M_\rho=(\D\times\C)/\Gamma$, with $\gamma(z,t)=(\gamma z,t+\rho(\gamma))$, is Stein, $X=\partial_t$ induces a free proper $\C$-action, and the descending $dt$ is closed and normalized. No normalized closed form on $M_\rho$ has discrete periods, by the Hodge-theoretic argument in that theorem.
\end{example}

Discrete-period approximation
may also occur for nonperiodic additive flows:

\begin{example}[Nonperiodic additive flow with discrete-period approximation]\label{ex:additive-success}
On $M=(\C^*)^2$, let
\[
 X=z\partial_z+\sqrt2\,w\partial_w,\qquad
 \omega=\frac12\frac{dz}{z}+\frac1{2\sqrt2}\frac{dw}{w}.
\]
The complete flow $\Phi_t(z,w)=(e^tz,e^{\sqrt2 t}w)$ defines a faithful additive $\C$-action (no nonzero common time period). For real $\varepsilon\ne0$ set
\[
 \omega_\varepsilon=\Big(\tfrac12+i\varepsilon\Big)\frac{dz}{z}
 +\frac{\tfrac12-i\varepsilon}{\sqrt2}\frac{dw}{w}.
\]
Then $d\omega_\varepsilon=0$, $\omega_\varepsilon(X)=1$, $\omega_\varepsilon\to\omega$, and the two period generators are $\R$-linearly independent for $\varepsilon\ne0$. Thus period groups are lattices. This illustrates why the genus-three obstruction is genuinely geometric rather than a consequence of the absence of a circle subgroup.
\end{example}

\section{Singular multiplicative actions and logarithmic time}\label{sec:singular-log}
When the infinitesimal generator has zeros, logarithmic forms provide normalized time coordinates away from invariant divisors.

In this subsection $X$ is \emph{allowed to vanish}, so $\T_X$ need not exist on the whole manifold. A meromorphic normalization $\eta(X)=1$ is interpreted as an identity in the field of meromorphic functions, not as a pointwise equality at the fixed points.

\begin{theorem}[A logarithmic normalized time form for a singular action]\label{stein:thm:singular}
Let $M$ be a connected Stein manifold and let $X\not\equiv0$ be the holomorphic infinitesimal generator of a holomorphic $\C^*$-action, possibly with fixed points. Then there are $k\in\Z\setminus\{0\}$ and a nonzero holomorphic semi-invariant $f\in\OO(M)$ satisfying $X(f)=k f$. Its zero divisor $D=(f=0)$ is invariant and contains $\operatorname{Sing}(X)$. The meromorphic form
\[
 \eta=\frac1k\frac{df}{f}
\]
is closed and satisfies $\eta(X)=1$ meromorphically. The complement $U=M\setminus D$ is Stein; $X|_U$ is nowhere zero and complete; the restriction of $\eta$ is normalized, holomorphic and closed with
\[
 \Per_U(\eta)\subseteq \frac{2\pi i}{k}\Z.
\]
Moreover $f:U\to\C^*$ is a surjective holomorphic submersion. If $H_1(U,\Z)$ is finitely generated, each closed holomorphic $\omega$ on $U$ with $\omega(X)=1$ is a compact-open limit on $U$ of normalized closed forms with discrete period groups, and each such approximant defines a holomorphic suspension over $\C/\Per_U(\omega_j)$.
\end{theorem}
\begin{proof}
Choose $p\in M$ with $X(p)\ne0$. By holomorphic separability of tangent vectors on a Stein manifold choose $h\in\OO(M)$ such that $X(h)(p)\ne0$. Its circle Fourier coefficients
\[
 h_k(q)=\frac1{2\pi}\int_0^{2\pi}e^{-ik\theta}h(e^{i\theta}\cdot q)\,d\theta,\qquad k\in\Z,
\]
are holomorphic and satisfy $h_k(s\cdot q)=s^k h_k(q)$ for all $s\in\C^*$, by holomorphic continuation from $S^1$. Since $X(h)(p)\ne0$, at least one coefficient of nonzero weight is nonzero. Take $f=h_k$ for such $k$.
The semi-invariance yields $X(f)=kf$ and implies that its zero divisor is invariant and contains every zero of $X$. On the meromorphic level,
$d(df/f)=0$ and $(df/f)(X)=X(f)/f=k$. The complement $U$ is biholomorphic to the closed submanifold
$\{(p,v)\in M\times\C:vf(p)=1\}$ of the Stein manifold $M\times\C$, hence is Stein. Completeness and preservation of $U$ follow from the original $\C^*$-action. Since $X(f)=kf$ and $k\ne0$, $f$ is a submersion on $U$; the nonempty image is a $\C^*$-orbit under $f(s\cdot p)=s^k f(p)$, hence is all of $\C^*$.
For any closed path $\gamma$ in $U$, the winding number of $f\circ\gamma$ is integral and
$\int_\gamma\eta=(2\pi i/k)\operatorname{wind}(f\circ\gamma,0)$. Thus the periods lie in the displayed discrete group. Finally apply the nonsingular normalized approximation Theorem~\ref{stein:thm:Cstar} to $U$ under the stated finite-generation hypothesis, followed by the suspension Theorem~\ref{stein:thm:suspension}. The fibers of $f$ itself need not be connected; the suspension associated with the full period group supplies connected fibers.
\end{proof}

The logarithmic construction produces a corresponding statement for singular multiplicative actions.

\begin{theorem}[Singular multiplicative actions]\label{thm:intro:e}
Let $M$ be a connected Stein manifold and $X\not\equiv0$ the infinitesimal generator of a holomorphic $\C^*$-action, possibly with zeros. There exist $k\in\Z\setminus\{0\}$ and a nonzero holomorphic semi-invariant $f$ with $X(f)=kf$. The closed meromorphic form $k^{-1}df/f$ satisfies $\eta(X)=1$ meromorphically. Its polar divisor is invariant and contains $\operatorname{Sing}(X)$; on its Stein complement $U$ the form is holomorphic, normalized and has discrete periods contained in $(2\pi i/k)\Z$. Under finite generation of $H_1(U,\Z)$, all normalized closed forms on $U$ admit discrete-period approximation. 
\end{theorem}

\begin{proof}[Proof of Theorem~\ref{thm:intro:e}]
Theorem~\ref{stein:thm:singular} constructs a nonzero
semi-invariant holomorphic function $f$ of nonzero integral
weight $k$ and establishes the normalized logarithmic form
$k^{-1}df/f$, its invariant polar divisor, Stein complement,
and discrete period containment. Its final assertion follows
by applying normalized approximation on that complement,
provided its first integral homology is finitely generated.
These conclusions give Theorem~\ref{thm:intro:e}.
\end{proof}

Example~\ref{ex:singular-log} below displays the logarithmic pole at a fixed point.

\begin{example}[The logarithmic pole at a fixed point]\label{ex:singular-log}
For the action $s\cdot z=sz$ on $M=\C$, the generator $X=z\partial_z$ vanishes at the origin. The semi-invariant $f=z$ has weight $1$, and $\eta=dz/z$ is a closed meromorphic form with $\eta(X)=1$. On $U=\C^*$, its period group is $2\pi i\Z$, even though $H^1(\C,\C)=0$. Thus the relevant cohomology is $H^1(U,\C)$ and logarithmic residues matter.
\end{example}
\begin{remark}[Meromorphic versus merely holomorphic approximation]
The theorem guarantees approximation \emph{on $U$}. It does not assert that arbitrary approximants extend meromorphically across the chosen divisor with a prescribed pole order. Such an algebraic strengthening can be studied separately using algebraic de Rham representatives, but is not assumed in the analytic result.
\end{remark}

\section{Hyperbolicity and Period Rigidity in Holomorphic Flows}
\label{sec:hyperbolic-returns}
Finally, transverse hyperbolicity together with global dynamical product structure supplies a dynamical mechanism for discreteness of periods distinct from approximation.

The connection between jointly integrable Anosov dynamics and invariant
closed one-forms goes back to Plante's treatment of smooth Anosov flows
\cite{Plante1972} and, in the holomorphic setting, to
\cite[\S4]{Scardua2003}. The latter develops the compact complex theory,
whereas the arguments below concern complete complex-time flows on
noncompact manifolds. We use \emph{holomorphic hyperbolic splitting}
in a specified sense, without identifying this definition with every
notion of Anosov flow in the compact holomorphic literature.

\begin{definition}[Jointly integrable holomorphic hyperbolic flow]
\label{def:joint-hyperbolic}
Let $M$ be a connected complex manifold and let $X$ be a nowhere-vanishing
complete holomorphic vector field, with complex-time flow
$\Phi:\C\times M\to M$. We say that $\Phi$ is a \emph{jointly integrable
holomorphic hyperbolic flow} if there are holomorphic subbundles
$E^s,E^u\subset TM$, both of positive complex rank and invariant under
$D\Phi_t$ for every $t\in\C$, such that
\begin{equation}\label{eq:hyp-splitting}
 TM=E^s\oplus\C X\oplus E^u,\qquad E^{su}:=E^s\oplus E^u
\end{equation}
is an involutive holomorphic distribution, and if there exist a Hermitian
metric $h$ on $M$ and constants $C,\lambda>0$ for which, for all real
$t\ge0$,
\begin{equation}\label{eq:hyp-estimates}
 \|D\Phi_t(v_s)\|_h\le Ce^{-\lambda t}\|v_s\|_h,
 \qquad
 \|D\Phi_{-t}(v_u)\|_h\le Ce^{-\lambda t}\|v_u\|_h
\end{equation}
for $v_s\in E^s$ and $v_u\in E^u$. We call $E^s$ and $E^u$
the stable and unstable distributions, and the integral foliation
of $E^{su}$ the joint stable--unstable foliation.
\end{definition}

These are uniform \emph{real-time} estimates for a flow complete in
\emph{complex time}. In the noncompact setting the metric is part of the
data: the estimates are not invariant under an arbitrary change of
Hermitian metric. We do not require completeness of $h$ in the
definition, although both positive examples below admit complete
metrics of the indicated kind. The term ``hyperbolic'' is used in this
specified sense, without asserting equivalence with all definitions
of holomorphic Anosov flow on compact complex manifolds.

\begin{lemma}[The normalized joint time form]\label{lem:anosov-time}
Under the invariant integrable splitting~\eqref{eq:hyp-splitting}
there is a unique holomorphic one-form $\omega$ satisfying
$\omega(X)=1$ and $\ker\omega=E^{su}$. It is closed and invariant
under every $\Phi_t$, $t\in\C$.
\end{lemma}
\begin{proof}
The decomposition of the tangent bundle defines a holomorphic
projection to the line generated by $X$, hence the indicated form.
Since the splitting is invariant, $\Phi_t^*\omega=\omega$ and
$\mathcal L_X\omega=0$. Cartan's formula and the normalization give
$\iota_Xd\omega=0$. If $Y,Z$ are local sections of $E^{su}$,
involutivity gives $d\omega(Y,Z)=-\omega([Y,Z])=0$.
These identities exhaust the components of $d\omega$ under the
direct-sum decomposition of $TM$.
\end{proof}

For $p\in M$ let $L_p$ denote the leaf of $\ker\omega$ through $p$.
If $\tau\in\C\setminus\{0\}$ satisfies $\Phi_\tau(p)=p$, then the
orbit loop has $\omega$-integral $\tau$, so that
$\tau\in\Lambda:=\Per(\omega)$. Therefore $\Phi_\tau$ preserves
$L_p$, and its restriction
$P_\tau=\Phi_\tau|_{L_p}$ is a holomorphic return automorphism.
The relation between a period of the time form and a return time of a
point is central to the next proposition.

\begin{proposition}[Isolated-return criterion]\label{prop:isolated-return}
Let $M$ be a connected complex manifold, $X$ a complete holomorphic
vector field and $\omega\in\mathcal T_X$. Suppose
$\Phi_\tau(p)=p$ for some $p\in M$ and $\tau\ne0$. Suppose there
is a neighborhood $U$ of $p$ in $M$ such that
\[
 U\cap L_p\cap\operatorname{Fix}(\Phi_\tau)=\{p\}.
\]
Then $\Per(\omega)$ is discrete. Accordingly the foliation $\ker\omega$
is the foliation by connected fibers of a holomorphic fiber bundle
onto $\C/\Per(\omega)$.
\end{proposition}
\begin{proof}
Put $\Lambda=\Per(\omega)$. For every $s\in\Lambda$, the time-$s$
map preserves $L_p$. Commutativity of the flow implies
\[
 P_\tau(\Phi_s(p))=\Phi_s(\Phi_\tau(p))=\Phi_s(p),
\]
so $\Phi_s(p)$ is a fixed point of $P_\tau$.
If $\Lambda$ were not discrete, it would contain nonzero elements
$s_j\to0$. Then $\Phi_{s_j}(p)\to p$ in $M$; the ambient
isolation hypothesis forces
$\Phi_{s_j}(p)=p$ for large $j$. But the isotropy subgroup
$\{t:\Phi_t(p)=p\}$ is discrete because $X(p)\ne0$, a
contradiction. The bundle conclusion is
Theorem~\ref{stein:thm:suspension}.
\end{proof}

The return map at a real hyperbolic periodic orbit gives a complementary period-discreteness criterion under the same ambient isolation hypothesis.

\begin{theorem}[Discreteness from a real hyperbolic periodic orbit with ambient isolation]
\label{thm:real-hyp-period}
Let $\Phi$ be a jointly integrable holomorphic hyperbolic flow in the sense of Definition~\ref{def:joint-hyperbolic}, and let
$\omega$ be the form from Lemma~\ref{lem:anosov-time}. If there
exists $p\in M$ with $\Phi_T(p)=p$ for a real $T>0$ and an
ambient neighborhood $U$ of $p$ satisfying
\[
 U\cap L_p\cap\operatorname{Fix}(\Phi_T)=\{p\},
\]
then $\Per(\omega)$ is discrete and $\ker\omega$ is a holomorphic
fiber-bundle foliation over $\C/\Per(\omega)$.
If $M$ is Stein, the base is $\C^*$ or an elliptic curve.
\end{theorem}
\begin{proof}
Write $P=\Phi_T|_{L_p}$ and
$A=D\Phi_T(p)|_{T_pL_p}$. The invariant splitting of $T_pL_p$
and the estimates for $D\Phi_{nT}(p)=A^n$ show that the spectrum of
$A$ on $E^s_p$ has modulus less than one. Applying the same argument
to $A^{-n}$ on $E^u_p$ shows that the spectrum on $E^u_p$ has
modulus greater than one. In particular, $A-I$ is invertible.
The holomorphic inverse function theorem applied to $P-\mathrm{id}$
in a chart of the leaf at $p$ makes $p$ an isolated fixed point of
$P$ in the intrinsic leaf topology. This alone does not imply ambient
isolation. The additional ambient hypothesis permits application of
Proposition~\ref{prop:isolated-return}, giving discreteness and the
bundle structure. Finally $T\in\Per(\omega)$ rules out the zero
subgroup, leaving rank one or two.
\end{proof}

The hyperbolicity estimates concern real time. The ambient-isolation
criterion itself applies to every nonzero complex return time.
A stronger conclusion is
available when the global stable and unstable manifolds of the return
map have a product property on the actual joint leaf, rather than
only on a covering space. This global product hypothesis does not
require the ambient isolation assumption of
Theorem~\ref{thm:real-hyp-period}.

\begin{theorem}[Cyclicity under global dynamical product structure]
\label{thm:global-product-cyclic}
Let $\Phi$ satisfy Definition~\ref{def:joint-hyperbolic}, let
$\omega$ be its normalized joint time form, and suppose that
$\Phi_T(p)=p$ for some real $T>0$. Put $P=\Phi_T|_{L_p}$.
Suppose that for every $a,b\in L_p$ the global dynamical stable manifold
$W^s_P(a)$ and global dynamical unstable manifold $W^u_P(b)$
intersect in exactly one point. Then
\[
 \Per(\omega)=\{t\in\C:\Phi_t(p)=p\}.
\]
If $M$ is Stein, this common group is a nonzero discrete cyclic
subgroup of $\C$, and the joint foliation defines a holomorphic
fiber bundle over $\C^*$.
\end{theorem}
\begin{proof}
Suppose $q\in L_p$ is fixed by $P=\Phi_T|_{L_p}$. The point
$r\in W^s_P(p)\cap W^u_P(q)$ is unique. Since $P$ preserves these
two sets, uniqueness implies $P(r)=r$. Its positive iterates converge
to $p$, and its negative iterates converge to $q$. As $r$ is fixed,
$r=p=q$. Therefore $\operatorname{Fix}(P)=\{p\}$.
For every $s\in\Per(\omega)$, the preceding commuting-return
identity makes $\Phi_s(p)$ a fixed point of $P$. It follows that
$\Phi_s(p)=p$. Conversely, integration along an orbit loop shows
that every element of the isotropy group of $p$ belongs to
$\Per(\omega)$. The two groups coincide.

The common group is discrete because $X(p)\ne0$, so its isotropy
has no accumulation at the origin; it is nonzero because it contains
$T$. A discrete rank-two isotropy lattice would make the
orbit through $p$ a compact elliptic curve $\C/\Lambda$,
embedded as a complex orbit in $M$. This is impossible when $M$ is
Stein: every holomorphic map from a compact connected complex curve
to a Stein manifold is constant. Thus the isotropy, and hence the
period group, is cyclic. The suspension principle completes the proof.
\end{proof}

\subsection{Proof of Theorem~\ref{thm:intro:hyperbolic}}
Now we are ready to prove the last of our main results.

\begin{proof}[Proof of Theorem~\ref{thm:intro:hyperbolic}]
The invariant splitting, involutivity, and hyperbolic estimates
are precisely the jointly integrable hypotheses of
Definition~\ref{def:joint-hyperbolic}. Lemma~\ref{lem:anosov-time}
constructs the unique normalized closed holomorphic time form
with kernel $E^s\oplus E^u$. The real-period and global
stable--unstable product hypotheses allow
Theorem~\ref{thm:global-product-cyclic} to be applied. It identifies
the period group with the nonzero discrete isotropy of $p$, and
Theorem~\ref{stein:thm:suspension} gives the holomorphic fiber bundle.
When $M$ is Stein the group is cyclic, so its quotient base is
$\C^*$, completing the proof.
\end{proof}

Here $W^s_P(a)$ and $W^u_P(b)$ denote the global
dynamical stable and unstable sets of $P$, respectively,
defined by
\[
W^s_P(a)=\{z\in L_p:d(P^n(z),P^n(a))\to0
\text{ as }n\to+\infty\},
\]
and
\[
W^u_P(b)=\{z\in L_p:d(P^{-n}(z),P^{-n}(b))\to0
\text{ as }n\to+\infty\},
\]
where $d$ is a compatible metric.

\subsection{Examples}
Next we give a holomorphic flow
with hyperbolic return dynamics and a cyclic period group.
\begin{example}
\label{ex:hyperbolic-Cstar}
Take $M=\C^*\times\C^2$ with coordinates $(q,x,y)$, choose $a,T>0$,
and let
\[
 X=\frac{2\pi i}{T}q\partial_q-ax\partial_x+ay\partial_y,
 \qquad \omega=\frac{T}{2\pi i}\frac{dq}{q}.
\]
The complete flow is
\[
 \Phi_t(q,x,y)=(e^{2\pi it/T}q,e^{-at}x,e^{at}y),
\]
so $d\omega=0$, $\omega(X)=1$, and
$\Per(\omega)=T\Z$. The joint foliation consists of the
fibers $\{q\}\times\C^2$. The point $(1,0,0)$ has real period
$T$, and its return on the fiber is
$P_T(x,y)=(e^{-aT}x,e^{aT}y)$, with unique fixed point.
The invariant bundles $E^s=\C\partial_x$ and
$E^u=\C\partial_y$ satisfy~\eqref{eq:hyp-estimates} for the
complete Hermitian metric
\[
 h=\left|dx+\frac{aT}{2\pi i}x\frac{dq}{q}\right|^2
  +\left|dy-\frac{aT}{2\pi i}y\frac{dq}{q}\right|^2
  +\left|\frac{dq}{q}\right|^2.
\]
Indeed the three displayed coframe elements are respectively
multiplied by $e^{-at},e^{at}$, and $1$ under pullback by $\Phi_t$.
For completeness, lift a rectifiable finite-length path to a
logarithmic coordinate $u$ with $q=e^u$. The $|du|^2$ term gives a finite
variation of $u$, so its image is bounded. Along the path, the other
coframe terms are $dx+(aT/(2\pi i))x\,du$ and
$dy-(aT/(2\pi i))y\,du$. The corresponding inhomogeneous linear
integral inequalities, with bounded $u$ and finite integrals of the
coframe norms, show via Gronwall's inequality that $x$ and $y$ are
bounded and Cauchy toward a finite-length endpoint. Thus no divergent
curve has finite length, which proves completeness. This furnishes a noncompact Stein model with both stable and unstable directions and a genuine real periodic orbit. In particular, it is a jointly integrable holomorphic hyperbolic flow in the sense of Definition~\ref{def:joint-hyperbolic}.
\end{example}

\begin{example}[Nonperiodic jointly integrable hyperbolic flow]
\label{ex:hyperbolic-affine}
Let $M=\C^3$ with coordinates $(x,y,z)$, choose $a>0$, and set
\[
X=\partial_z-ax\partial_x+ay\partial_y,\qquad \omega=dz.
\]
The complete complex-time flow is
\[
\Phi_t(x,y,z)=(e^{-at}x,e^{at}y,z+t).
\]
The translation in $z$ shows that every orbit is nonperiodic and
biholomorphic to $\C$. The invariant holomorphic bundles
$E^s=\C\partial_x$ and $E^u=\C\partial_y$ have positive rank and
$E^{su}=\ker dz$ is integrable. Furthermore $d\omega=0$,
$\omega(X)=1$, and $\Per(\omega)=\{0\}$. Put
\[
\theta_s=dx+ax\,dz,\qquad
\theta_u=dy-ay\,dz,\qquad \theta_0=dz.
\]
These forms constitute a global holomorphic coframe, with
$\Phi_t^*\theta_s=e^{-at}\theta_s$,
$\Phi_t^*\theta_u=e^{at}\theta_u$, and
$\Phi_t^*\theta_0=\theta_0$. Thus the metric
$h=|\theta_s|^2+|\theta_u|^2+|\theta_0|^2$ satisfies
\eqref{eq:hyp-estimates} with $C=1$ and $\lambda=a$.
It is complete: in the global coordinates
$u=e^{az}x$, $v=e^{-az}y$, and $z$, it reads
\[
h=e^{-2a\operatorname{Re}z}|du|^2
  +e^{2a\operatorname{Re}z}|dv|^2+|dz|^2.
\]
A finite-length path has bounded $z$-variation, and on bounded
$z$-regions the metric is uniformly comparable with the Euclidean
metric in $(u,v,z)$. Hence no divergent path can have finite length.
This is a jointly integrable holomorphic hyperbolic flow in the sense
of Definition~\ref{def:joint-hyperbolic}, with an exact normalized
time form and no periodic complex-time orbit.
\end{example}

\begin{example}[Elliptic periods without hyperbolicity]
\label{ex:elliptic-neutral}
Let $M=(\C^*)^2\times\C^2$, with coordinates $(z,w,x,y)$, and
set
\[
 X=2\pi i z\partial_z-x\partial_x+y\partial_y,
 \qquad
 \omega=\frac{1}{2\pi i}\frac{dz}{z}
       +\frac{1}{2\pi}\frac{dw}{w}.
\]
Then $\Phi_t(z,w,x,y)=(e^{2\pi it}z,w,e^{-t}x,e^ty)$, and
$\omega(X)=1$. Direct integration over the two standard generators
of $H_1((\C^*)^2,\Z)$ gives
\[
 \Per(\omega)=\Z+i\Z.
\]
The point $(1,1,0,0)$ has real period $1$, but $w$ is constant
under the flow. The tangent direction $w\partial_w$ contributes a
neutral direction to $\ker\omega$ after subtraction of a multiple
of $X$, so the splitting~\eqref{eq:hyp-splitting} cannot be uniformly
hyperbolic on the entire joint distribution. This example shows that
Steinness, a closed normalized time form, and a real periodic orbit
alone do not exclude elliptic periods. The next example realizes
elliptic periods even with the full hyperbolicity of
Definition~\ref{def:joint-hyperbolic}.
\end{example}

\begin{example}[Hyperbolic returns with nondiscrete or elliptic periods]
\label{ex:hyperbolic-period-family}
Let $a>0$, $\alpha\in\C$, and put
\[
 V=\{(x,y)\in\C^2:xy\ne1\},\qquad
 M=\C^*\times V,\qquad f=xy-1.
\]
This is a connected smooth affine Stein threefold. Indeed it is
biholomorphic to the closed affine submanifold
\[
 \{(q,r,x,y,v)\in\C^5:qr=1,\ v(xy-1)=1\}.
\]
Set
\[
 X=2\pi i q\partial_q-ax\partial_x+ay\partial_y,
 \qquad
 \omega=\frac1{2\pi i}\left(\frac{dq}{q}
                         +\alpha\frac{df}{f}\right).
\]
The complete complex-time flow is
\[
 \Phi_t(q,x,y)=(e^{2\pi it}q,e^{-at}x,e^{at}y).
\]
Since $f$ is invariant, $d\omega=0$ and $\omega(X)=1$.
The loop around the $q$-factor has period $1$, while the loop
$q=1$, $x=1$, $y=1+\varepsilon e^{i\theta}$, for
$0\le\theta\le2\pi$ and $\varepsilon>0$, has period $\alpha$.
Every period is an integral combination of these values, by the
winding-number formula. Consequently
\[
 \Per(\omega)=\Z+\alpha\Z.
\]

Define global holomorphic vector fields
\[
 V_s=\partial_x-\alpha\frac{qy}{f}\partial_q,
 \qquad
 V_u=\partial_y-\alpha\frac{qx}{f}\partial_q.
\]
They span line bundles $E^s=\C V_s$, $E^u=\C V_u$, with
\[
 TM=E^s\oplus\C X\oplus E^u,
 \qquad E^s\oplus E^u=\ker\omega.
\]
The joint distribution is involutive because $\omega$ is closed,
and direct differentiation gives
\[
 D\Phi_t(V_s|_p)=e^{-at}V_s|_{\Phi_t(p)},\qquad
 D\Phi_t(V_u|_p)=e^{at}V_u|_{\Phi_t(p)}
 \quad(t\in\C).
\]
The holomorphic coframe
\[
 \theta_s=dx+ax\omega,\qquad
 \theta_u=dy-ay\omega,\qquad \omega
\]
is dual to $(V_s,V_u,X)$ and satisfies
\[
 \Phi_t^*\theta_s=e^{-at}\theta_s,\qquad
 \Phi_t^*\theta_u=e^{at}\theta_u,\qquad
 \Phi_t^*\omega=\omega.
\]
Thus $h=|\theta_s|^2+|\theta_u|^2+|\omega|^2$ is a Hermitian
metric satisfying~\eqref{eq:hyp-estimates} with $C=1$ and
$\lambda=a$. This verifies all the hypotheses of
Definition~\ref{def:joint-hyperbolic}, which does not require
completeness of $h$.

The point $p=(1,0,0)$ has real period $1$. On the joint leaf
through $p$, the derivative of $P=\Phi_1|_{L_p}$ at $p$ has
eigenvalues $e^{-a}$ and $e^a$. In particular, $p$ is an isolated
fixed point in the intrinsic leaf topology. Nevertheless, if
$\alpha\in\R\setminus\Q$, then $\Per(\omega)$ is nondiscrete.
For nonzero $s_j\in\Z+\alpha\Z$ tending to zero, the distinct
points $\Phi_{s_j}(p)=(e^{2\pi i s_j},0,0)$ lie on $L_p$,
are fixed by $P$, and converge to $p$ in $M$. They do not converge
in the intrinsic leaf topology: a local plaque through $p$ is
given by branches of $\log q+\alpha\log f=\mathrm{constant}$,
and meets $\{x=y=0\}$ only at $p$.

If $\operatorname{Im}\alpha\ne0$, then $\Per(\omega)$ is a
rank-two lattice, so the joint foliation is a holomorphic fiber
bundle over an elliptic curve. For $\alpha=i$, the points
$\Phi_{ni}(p)=(e^{-2\pi n},0,0)$, $n\in\Z$, are distinct
isolated fixed points of $P$ on the same joint leaf. Thus elliptic
periods can occur under full hyperbolicity as defined here.
\end{example}

The following example exhibits a holomorphic flow
with hyperbolic stable and unstable directions
whose joint distribution $E^s\oplus E^u$ is
contact and hence nonintegrable.
\begin{example}[A complete hyperbolic Reeb flow on an affine Stein group]
\label{ex:sl2-contact}
Let $G=\mathrm{SL}(2,\C)$, the smooth affine Stein hypersurface
$ad-bc=1$ in $\C^4$. In its Lie algebra choose
\[
H=\frac12\begin{pmatrix}1&0\\0&-1\end{pmatrix},\quad
E=\begin{pmatrix}0&1\\0&0\end{pmatrix},\quad
F=\begin{pmatrix}0&0\\1&0\end{pmatrix}.
\]
Then $[H,E]=E$, $[H,F]=-F$, and $[E,F]=2H$. Define
\[
\Phi_t(g)=g\exp(tH),\qquad t\in\C.
\]
This is a complete holomorphic flow. Write $H^L,E^L,F^L$ for the
left-invariant holomorphic vector fields ($A^L_g=d(L_g)_eA=gA$).
The infinitesimal generator is $X=H^L$. Right multiplication acts in
the left trivialization by $\operatorname{Ad}(\exp(-tH))$, hence
\[
D\Phi_t(E^L_g)=e^{-t}E^L_{\Phi_t(g)},\qquad
D\Phi_t(F^L_g)=e^tF^L_{\Phi_t(g)}.
\]
For the left-invariant Hermitian metric $h$ making $H^L,E^L,F^L$
orthonormal, these are exact real-time hyperbolicity estimates with
$C=\lambda=1$. The underlying left-invariant Riemannian metric is
complete, as is every left-invariant Riemannian metric on a connected
Lie group. The line fields $E^s=\C E^L$ and $E^u=\C F^L$ are
invariant and holomorphic, but their sum is not integrable, because
$[E^L,F^L]=2H^L$.

Let $\alpha$ be the left-invariant holomorphic one-form defined by
$\alpha(H^L)=1$, $\alpha(E^L)=\alpha(F^L)=0$. The Maurer--Cartan
formula yields
\[
d\alpha(E^L,F^L)=-\alpha([E^L,F^L])=-2,
\qquad \alpha\wedge d\alpha\ne0.
\]
Also $\alpha(X)=1$ and $\iota_Xd\alpha=0$, so $X$ is the
holomorphic Reeb vector field of the contact form $\alpha$.
In contrast to the jointly integrable case, no closed holomorphic
one-form $\eta$ on $G$ can satisfy $\eta(X)=1$. Indeed,
$\exp(4\pi iH)=I$, and the associated complex-time orbit loop
would have $\eta$-integral $4\pi i$; this contradicts the fact that
$G$ is simply connected, hence that every closed holomorphic
one-form is exact. Thus $\mathcal T_X=\varnothing$.
The real subgroup $\mathrm{SL}(2,\R)$ carries the corresponding
homogeneous real-time geodesic-flow model, up to its standard
central quotient. This familiar homogeneous construction is used
here as a geometric contrast, not asserted to be original.
\end{example}

Thus, uniform real-time hyperbolicity alone
does not imply the joint integrability of
$E^s\oplus E^u$ required in
Definition~\ref{def:joint-hyperbolic}.

\section*{Statements and Declarations}
\noindent\textbf{Funding.} This research was supported by Fundação Getúlio Vargas (FGV-EMAp, Rio de Janeiro). 

\noindent\textbf{Competing interests.} There are no conflicts of interest.
\noindent\textbf{Data availability.} No datasets were generated or analysed in this theoretical mathematical study.

\noindent\textbf{Use of generative AI tools.}  The author used ChatGPT (OpenAI) 
 to assist with language editing, LaTeX compilation and text grammar review.
All mathematical statements, proofs, and references
are due to and were reviewed by the author, who takes full
responsibility for the accuracy, originality,
and final content of the manuscript.

\end{document}